\documentclass[a4 paper, 12 pt]{article}

\usepackage{amsthm,amssymb,amsmath,amsfonts,mathrsfs,amscd}
\usepackage[latin1]{inputenc}
\usepackage{verbatim}

\theoremstyle{definition}
\newtheorem{assumption}{Assumption}[section]

\newtheorem{defi}[assumption]{Definition}

\theoremstyle{remark}
\newtheorem{rem}[assumption]{Remark}

\theoremstyle{plain}
\newtheorem{lemma}[assumption]{Lemma}
\newtheorem{coro}[assumption]{Corollary}
\newtheorem{theorem}[assumption]{Theorem}
\newtheorem{proposition}[assumption]{Proposition}

\newcommand{\End}{\mathrm{End}}

\newcommand{\norm}[1]{\mathrm N(#1)}
\newcommand{\val}{\mathrm{val}}

\newcommand{\Ker}{\mathrm{Ker}}
\newcommand{\Hom}{\mathrm{Hom}}

\newcommand{\Image}{\mathrm{Im}}

\newcommand{\GL}{\mathrm{GL}}
\newcommand{\M}{\mathrm{M}}

\newcommand{\Ext}{\rom{Ext}}

\newcommand{\cyc}{{\rom{cyc}}}
\newcommand{\cpt}{\mathrm{cpt}}

\newcommand{\Q}{\mathbb Q}

\newcommand{\Z}{\mathbb Z}
\newcommand{\R}{\mathbb R}
\newcommand{\C}{\mathbb C}

\newcommand{\fr}{\mathfrak}
\newcommand{\cl }{\mathcal}
\newcommand{\bk}{\backslash}
\newcommand{\wt}{\widetilde}

\newcommand{\mat}[4]{\left(\begin{array}{cc}#1&#2\\#3&#4\end{array}\right)}
\newcommand{\smallmat}[4]{\left(\begin{smallmatrix}#1&#2\\#3&#4\end{smallmatrix}\right)}
\newcommand{\dirlim}{\mathop{\varinjlim}\limits}
\newcommand{\invlim}{\mathop{\varprojlim}\limits}

\def \o {\rom{ord}}

\def \e {\mathfrak e}
\def \ec {\overline{\e}}

\newcommand \ord {{\mathrm{n,ord}}}
\newcommand{\rom}{\mathrm}

\newcommand{\W}{\mathbb W}

\newcommand{\sgn}{\rom{sgn}}

\def \r {\mathfrak r}
\def \n {\mathfrak n}
\def \m {\mathfrak m}

\def \q {\mathfrak q}

\def \frakp {\fr p}
\def \F {F}

\def \F {F}

\def \D {\mathbb{D}}

\def \Gbar {\overline{\Gamma}}

\def \Zbar {\overline{Z}}

\def \Gbar {\overline{\Gamma}}

\def \qpbar {\overline{\Q}_p}

\def \scrf {\mathcal{F}}

\def \scrh {\cl H}
\def \scrd {\mathcal{D}}

\def \scrr {\mathcal{R}}

\def \scro {\cl O}

\def\Q{\mathbb{Q}}

\def\Z{\mathbb{Z}}

\def\part{\partial}

\def\bp{\begin{pmatrix}}
\def\ep{\end{pmatrix}}
\def\bn{\begin{enumerate}}  
\def\en{\end{enumerate}}     
\def\Hom{\mbox{Hom}}
\def\ba{\begin{array}}  
\def\ea{\end{array}}

\begin{document}

\title{$\Lambda$-adic modular symbols and several variable $p$-adic
L-functions over totally real fields}
\author{B. Balasubramanyam and M. Longo}
\date{}

\maketitle
\tableofcontents

\section{Introduction}

Let $F/\Q$ be a totally real field of degree $d$. Denote by $\r$ its
ring of integers. For any ring $A$, define $\widehat
A:=A\otimes_\Z\widehat\Z$, where $\widehat\Z$ is the profinite
completion of $\Z$. Fix a compact open subgroup $S\subseteq
\GL_2(\widehat F)$ such that $U_0(\fr n)\supseteq S\supseteq U_1(\fr
n)$ for some integral ideal $\fr n$ of $F$, where $U_0(\fr n)$
(respectively, $U_1(\fr n)$) are the usual congruence groups defined
in \S \ref{HMF}, Equations \eqref{U0} and \eqref{U1}. Let $p$ be a
rational prime prime to $2\fr n$ which does not ramify in $F$ and
define $S(p^\alpha):=S\cap U(p^\alpha)$ for any non negative integer
$\alpha$, where $U(p^\alpha)$ is defined in \S \ref{HMF}, Equation
\eqref{U}. Fix an embedding $\iota:\overline\Q\hookrightarrow
\overline\Q_p$ and a finite extension $K$ of $\Q_p$ containing
$\iota\circ\mu(F)$ for all archimedean places $\mu$ of $F$. Denote
by $\cal O$ the ring of integers of $K$.

Let $I$ denote the set of embeddings of $\F$ into $\C$. Let $n,v \in
\Z[I]$ be fixed weight vectors such that $n+2v \equiv 0 \mbox{ mod }
\Z t$, where $t=(1,...1) \in \Z[I]$ and let $k:=n+2t$ and
$w:=v+k-t$. Following \cite{H2}, we denote by
$h_{k,w}^\ord(S(p^\alpha),\cl O)$ the Hecke algebra over $\cl O$ for
the space of $p$-nearly ordinary Hilbert cusp forms of weight
$(k,w)$ and level $S(p^\alpha)$. In \S \ref{NOHA} we recall Hida's
construction of the universal $p$-odinary Hecke algebra
\[\cl R\simeq \invlim _\alpha h_{2t,t}^\ord(S(p^\alpha),\cl O).\]
This Hecke algebra is \emph{universal} in the sense that each nearly ordinary Hecke
algebra $h_{k,w}^\ord(S_0(p^\alpha),\epsilon,K)$ acting on the $K$-vector space of
cusp forms of weight $(k,w)$, level $S_0(p^\alpha)$ and
finite order character $\epsilon:S_0(p^\alpha)/S_1(p^\alpha)\to\C$
is isomorphic to a residue
algebra of $\cal R$. Here $S_0(p^\alpha):=S\cap U_0(p^\alpha)$ and
$S_1(p^\alpha):=S\cap U_1(p^\alpha)$. More precisely, let
$$G:=\invlim_\alpha S_0(p^\alpha)\r^\times/S(p^\alpha)\r^\times,$$
where $S_0(p^\alpha):=S\cap U_0(p^\alpha)$ and denote by $W$ the
free part of $G$. Then $\cl R$ has a natural
structure of $\wt\Lambda$-algebra, where  $\wt\Lambda:=\cl O[\![G]\!]$ is
the Iwasawa algebra of $G$ and there are isomorphisms:
\begin{equation}\label{intro-cont-th-hecke-alg}
\cl R_P/P\cl R_P\rightarrow h^{\ord}_{k,w}(S_0(p^\alpha),\epsilon,K)\end{equation}
for suitable ideals
$P$ of $\wt\Lambda$. See \cite[Theorem 2.4]{H2} or \S \ref{NOHA}
for details.

Denote ${\cal L}:={\rm
Frac}(\Lambda)$, the field of fractions of $\cal L$ and by $\bar{
\cal L}$, its algebraic closure. Define \[\cl X(\bar {\cl L}):=\Hom_{\textrm{cont}}^{\cl O{\textrm{
-algebras}}}(\cl R,\bar{\cal L})\] and say that a point $\kappa\in\cl X(\bar{\cl L})$ is
\emph{arithmetic} it its restriction to $G$ has kernel equal to $P$ for
some of the ideals $P$ appearing in \eqref{intro-cont-th-hecke-alg}.
See \cite[pages 150-152]{H2}
or \S \ref{NOHA} for a precise notion of {arithmetic} point.
As a corollary of \cite[Theorem 2.4]{H2} is that points in
 correspond to $p$-adic families
of Hilbert modular forms interpolating classical Hilbert modular
forms of level $S(p^\alpha)$. More precisely, fix a point
$\theta\in\cl X(\bar{\cl L})$. By \cite[Theorem 2.4]{H2}, the image
of $\theta$ is contained in a finite extension $\cl K$ of $\cl L$.
Let $\cl I$ denote the integral closure of $\Lambda$ in $\cl K$.
Then each arithmetic point $\kappa\in\Hom_{\rm cont}^{\cl
O\rm{-algebras}}(\cl I,\overline\Q_p)$ corresponds to a classical
$p$-nearly ordinary Hilbert modular form $f_\kappa$ of suitable
weight $(k,w)$, level $S_0(p^\alpha)$ and character
$\epsilon$. See \cite[Corollary 2.5]{H2} for details. In particular, to any arithmetic point $\kappa$
is associated a weight $(n,v)$ (or $(k,w)$), a level $S_0(p^\alpha)$
and a character $\epsilon$. Finally, introduce the notion of \emph{primitive}
arithmetic point, which will be needed in the statement of Theorem
\ref{th-intro}, as in \cite[pages 317, 318]{H1}.

\begin{rem} Note that the Iwasawa algebra $\Lambda:={\cal
O}[\![W]\!]$ of $W$ is isomorphic to the formal power series ring in
$s$ variables ${\cal O}[\![X_1,\dots,X_s]\!]$, where
$s=d+1+\delta_F$ and $\delta_F$ is the non negative integer
appearing in Leopold's conjecture. An other way to state
\cite[Corollary 2.5]{H2} is that any $f\in S_{k,w}(U_1(\fr
np^\alpha),K)$ has $s$-dimensional $p$-adic \emph{deformations} over
$\cl O$  (in the sense of \cite[pages 152-153]{H2}).\end{rem}

To discuss the main result of this paper, we need some technical
assumptions. Denote by $F_\mathbb A$ the adele ring of $F$. By the
Strong Approximation Theorem, choose $t_\lambda\in\GL_2(F_\mathbb
A)$ for $\lambda\in\{1,\dots, h(p^\alpha)\}$ and a suitable integer
$h(p^\alpha)$ depending on $\alpha$ with $(t_\lambda)_{\fr
np}=(t_\lambda)_\infty=1$ and such that there is the following
disjoint union decomposition:
\[\GL_2(F_\mathbb
A)=\coprod_{\lambda=1}^{h(p^\alpha)}\GL_2(F)t_\lambda
S\R_+^\times\] where $\R_+$ is the set of positive real numbers
and $(t_\lambda)_{\fr np}$ (respectively, $(t_\lambda)_\infty$) is
the $\fr np$-part (respectively, the archimedean part) of
$t_\lambda$. If $\alpha=0$ write $h$ for $h(1)$. Define the
following arithmetic groups depending on ${S}$:
\[\Gamma^\lambda(p^\alpha):=x_\lambda{S}(p^\alpha
)x_\lambda^{-1}\cap\GL_2^+(F) \mbox{ and }
\Gbar^\lambda(p^\alpha):=\Gamma^\lambda(p^\alpha)/(\Gamma^\lambda(p^\alpha)\cap
F^\times),\] where $\GL_2^+(F)$ is the subgroup of $\GL_2(F)$
consisting of matrices with totally positive determinant.
\begin{assumption}\label{assumption-2}
The groups $\Gbar^\lambda:=\Gbar^\lambda(1)$ are torsion free for
all $\lambda=1,\dots,h$.
\end{assumption} See \cite[Lemma 7.1]{H1} for conditions under
which condition \ref{assumption-2} is verified. In particular, there
are infinitely many square-free integers for which
$\Gbar^\lambda(U_0(N))$ is torsion-free for all $\lambda$. Under
this assumption, for any $\Gamma^\lambda$-module $E$, there is a
canonical isomorphism $H^d(\Gamma^\lambda\bk\fr H^d,\wt E)\simeq
H^d(\Gbar^\lambda,E)$, where $\fr H$ is the complex upper half plane
and $\wt E$ is the coefficient system associated to $E$.

The aim of this work is to combine Hida theory in the totally real
case and Greenberg-Stevens theory of $\Lambda$-adic modular symbols
to obtain a $p$-adic $L$-function attached to a Hida family of
nearly ordinary Hilbert cusp forms. This $p$-adic $L$-function
interpolates $p$-adic $L$-functions attached to classical forms
$f_\kappa$ in the Hida families as in \cite[Theorem 5.15]{GS}. See
also the work of Kitagawa \cite{Ki} which inspired the construction
of the $p$-adic $L$ function in \cite{GS}. The main difference with
respect to the case of \cite{GS} relies on the fact that, as noticed
earlier, the Iwasawa algebra $\Lambda$ in this case is isomorphic to
a power series ring in \emph{at least} $d+1$-variables over $\cl O$,
while in the rational case it is just isomorphic to a power series
ring in \emph{one} variable. In \cite{H1} Hida constructs for each
weight $v$ as above a nearly ordinary universal Hecke algebra $\cl
R_v=h_{v}^\ord(U_1(p^\infty),\cl O)$ such that each
$h_{k,w}^\ord(U_1(p^\alpha),\cl O)$ with $k$ parallel to $v$ can be
obtained as a residue algebra of $\cl R_v$. These Hecke algebras
$\cl R_v$ are endowed with a structure of $\cl
O[\![X_1,\dots,X_{1+\delta_F}]\!]$-algebra. The Iwasawa algebra
$\Lambda$ considered here has more variables in order to unify these
various Hecke algebras as $v$ and the character $\epsilon$ vary.

A consequence of that fact that the Iwasawa algebra considered here
is bigger that that in \cite{GS} is that the role played by the set
of primitive vectors $(\Z_p^2)^\prime$ in \cite{GS} will be played
in this context by the $p$-adic space
\[X:=NC\bk\GL_2(\r_p)\simeq\invlim_\alpha S(p^\alpha)\r^\times\bk S\r^\times,\] which has a greater
rank. Here $N$ is the standard lower unipotent subgroup of
$\GL_2(\r_p)$ and $C$ is the closure of $\e:=S\cap F$ embedded
diagonally in $\GL_2(\r_p)$, where $\r_p:=\r\otimes_\Z\Z_p$. A
similar $p$-adic space has been defined and studied in \cite{as}.
The action of the Hecke operators $U_\fr p$ for prime ideals $\fr
p\mid p$ on $X$ is similar to that considered in \cite{as}.  To
describe $X$ more precisely, for any prime ideal $\fr p\mid p$ of
$F$, let $(\r_\fr p^2)^\prime$ denote the set of primitive vectors
of $\r_\fr p^2$, i. e. the set of elements $(x,y)\in\r_\fr p^2$ such
that at least one of $x$ and $y$ does not belong to $\fr p$. Set
$(\r_p^2)^\prime:=\prod_{\fr p\mid p}(\r_\fr p^2)^\prime$. Denote by
$\ec$ the closure of $S\cap F^\times$ in $\r_p^\times$. Then $X$ can
be identified via the map \[\gamma=\smallmat abcd\mapsto
((a,b),\det(\gamma))\] with
$\ec\bk((\r_p^2)^\prime\times\r_p^\times)$. Hence $X$ may be viewed
as an analogue of the primitive vectors appearing in \cite{GS}.
Elements of $X$ will be denoted by $(x,y),z)$.

Following \cite{GS}, define $\D_X$ to be the space of $\cl O$-valued
measures on $X$. This space is endowed with $\wt\Lambda$ and
$\Lambda$-algebra structures. Let $\cl D_X$ denote the local
coefficient system on $X_S$ associated to $\D_X$. We define the
space of $\Lambda$-adic modular symbols in this context to be
$$\W:=H_\cpt^d(X_S,\cl D_X)$$ the $d$-th cohomology with compact
supports of the Hilbert modular variety $X_S$ associated to $S$ with
coefficients in $\cl D_X$. It follows from \cite[Proposition
4.2]{AS} that this definition of $\Lambda$-adic modular symbols is
consistent with the analogous definition in \cite{GS}. The group
$\W$ is endowed with a structure of $\wt\Lambda$-module. It is also
endowed with an action of the involution $(\smallmat
100{-1},\dots,\smallmat 100{-1}) \in\GL_2(F)^d$, so there is a
notion of $\sgn$-eigenspace in $\W$ for each choice of $\sgn\in\{\pm
1\}^d$.

Let ${\rm Sym}^n(K)$ be the space of homogeneous polynomials with
coefficients in $K$ in $2d$ variables $X_\sigma,Y_\sigma$,
$\sigma\in I$, of degree $n_\sigma$ in $(X_\sigma ,Y_\sigma)$. For
any primitive arithmetic point $\kappa$ of weight $(n,v)$ and
character $\epsilon$, define a
specialization map $\rho_{\kappa}:{\D}_X\to {\rm Sym}^n(K)$ by
$$\mu \mapsto \rho_{\kappa}(\mu):=\int_{X^\prime}\epsilon(x)z^v (xY -
yX)^{n} d\mu(x,y,z),$$ where $X^\prime$ the subset of $X$ consisting
of elements $((x,y),z)$ such that $x\in\r_p^\times$. The map
$\rho_{\kappa}$ induces a map on the cohomology groups which we
again denote by the same symbol\[\rho_{\kappa} :\W\otimes_\Lambda\cl
R_{P}\to H^d_{\rm cpt}(X_{S_0(p^\alpha)},{\rm Sym}^n(K)).\] Here
$P:=\kappa_{|\Lambda}$ and $\cl R_{P}$ is the localization of $\cl
R$ at the kernel of $P:\Lambda\to\overline\Q_p$ (compare with
\eqref{intro-cont-th-hecke-alg}) and $S_0(p^\alpha)$ is the level of
the modular form $f_\kappa$ associated to the arithmetic point
$\kappa$. By following \cite{GS} and \cite{as}, a corresponding
control theorem for these specialization maps is stated in Theorem
\ref{main-theorem} and proved in \S \ref{proof-control-theorem}.
This result states that, for fixed $\theta:\cl R\to\cl I$ and
$\kappa_0\in\cl X(\cl I)$ an arithmetic point and for each choice of
sign $\sgn\in\{\pm 1\}^d$, there exists $\Phi\in (\W\otimes_{\wt
\Lambda}\cl R_{P_0})^\sgn$ (with ${\kappa_0}_{|\wt\Lambda}=P_0$)
such that
\begin{equation}\label{rel-mod-sym-intro}
\rho_{\kappa_0}(\Phi)=\Psi_{f_{\kappa_0}}^\sgn,\end{equation}
where $\Psi_{f_{\kappa_0}}^\sgn$ is
the $\sgn$-classical modular symbol associated to the cusp form $f_{\kappa_0}$ (see \S \ref{CMS}
for details and definitions). The interpolation formulas
satisfied by the classical modular symbols, which are recalled in \S
\ref{IFCMS}, lead to the following result:
\begin{theorem}\label{th-intro}Let $\theta$, $\kappa_0$ and $\Phi$ as above
(hence, \ref{rel-mod-sym-intro} holds). Choose a sign
$\epsilon\in\{\pm 1\}^d$. There exists a $p$-adic analytic function
$L_p^\sgn(\Phi,\theta,\kappa,\sigma)$ in the variables $\kappa\in
\cl X(\cl I)$ and $\sigma\in \cl X(\r_p^\times)={\rm Hom}_{\rm
cont}(\r_p^\times,\overline\Q_p^\times)$ such that for any primitive
arithmetic point $\kappa\in\cl X(\cl I)$ of weight $k$ and character
$\epsilon$, save possibly a finite number, any positive integer $m$
in the critical strip for the modular form $f_\kappa$ defined in
\eqref{critical-strip} and any primitive character $\chi$ of the
ideal class group of $F$ of conductor $p^m$ and of sign $\sgn$:
\[L_p^\sgn(\Phi,\theta,\kappa,\chi\chi_\cyc^{m-1})=
\left(1-\frac{\chi(p)\chi_\cyc^{m-1}(p)}{a_p(\theta,\kappa)}\right)
\frac{\Omega(\Phi,\theta,\kappa)}{\Omega^\sgn(f,\chi){a_p(\theta,\kappa)^m}}\Lambda(f_\kappa,\chi,m),\]
where $\chi_\cyc$ is the cyclotomic character,
$\Lambda(f_\kappa,\chi,m)$ is the complex $L$-function defined in \S
\ref{IFCMS}, Equations \eqref{Lambda} and \eqref{Lambda-1},
$a_p(\theta,\kappa):=\kappa\circ\theta(T(p))$ and finally
$\Omega(\Phi,\theta,\kappa)$ and ${\Omega^\sgn(f,\chi)}$ are
suitable periods defined in \eqref{eq-30} and
\eqref{Omega}.\end{theorem}

Let $f$ be an ordinary Hilbert modular form in the sense of
Panchishkin \cite[\S 8]{Pa}. It is well known that $f$ is nearly
ordinary in the sense of Hida. For any sign $\sgn\in\{\pm 1\}^d$,
denote by $L_p^\sgn(f,s)$ the $p$-adic $L$-function attached to $f$
and $\sgn$ constructed by Manin in \cite{Ma}. This function can be
characterized by its interpolation property: for any integer $m$ in
the critical strip of $f$,
\begin{equation}\label{int-form-intro}
L_p^\sgn(f,m)=\left(1-\frac{\chi(p)\chi_\cyc^{m-1}(p)}{a_p(\theta,\kappa)}\right)
\frac{\Lambda(f,\chi,m)}{\Omega^\sgn(f,\chi){a_p(\theta,\kappa)^m}}.\end{equation}
This formula can be found in a less explicit form in \cite[\S 5]{Ma}
or in \cite[Theorem 8.2]{Pa} in a form closest to this. The
interpolation formulas in Theorem \ref{th-intro} and
\eqref{int-form-intro}, the analiticity of the $p$-adic
$L$-functions and the fact that $L_p^\sgn(f,\chi)$ is uniquely
determined by \eqref{int-form-intro} by \cite[Theorem 8.2 (iii)]{Pa}
lead to the following:

\begin{coro}
Let $\Phi$, $\theta$ and $\kappa$ be as in Theorem \ref{th-intro}.
Suppose that $f_\kappa$ is ordinary in the sense of \cite{Pa}. Then
there is an equality of Iwasawa functions in $\sigma$:
\[L_p^\sgn(\Phi,\theta,\kappa,\sigma)=\Omega(\Phi,\theta,\kappa)L_p^\sgn(f_\kappa,\sigma).\]\end{coro}

In the spirit of \cite{GS}, the two variable $L$-function is a tool
to prove the exceptional zero conjecture for the derivative of the
one variable $p$-adic $L$-function associated to an elliptic curve
$E/\Q$ stated by Mazur, Tate and Teitelbaum in \cite{MTT}. The
natural development of this work is to investigate the analogue
conjecture for elliptic curves $E/F$ over totally real fields. An
interesting and new feature of the totally real context is that in
this case it may be possible to calculate partial derivatives
$\partial/\partial_\sigma$ with respect to each $\sigma\in I$. It is
expected that the parallel derivative
$\prod_\sigma\partial/\partial_\sigma$ will play the role of the
weight derivative in the context of \cite{GS}. It may also by
interesting to investigate the meaning of the non parallel
derivative operators and their connections with the geometry of the
elliptic curve $E/F$. We hope to come to this questions in a future
work.

\bigskip

\noindent{\bf Acknowledgments.} This paper is part of the Ph.D
thesis of the first author. The first author would like to thank
sincerely his advisor Prof. F. Diamond for his help and
encouragement. Both authors would like to thank Prof. H. Darmon for
suggesting the problem and for his advice and support during their
visits to McGill University.

\section{Hida's theory for Hilbert modular forms}\label{section-2}
\subsection{Hilbert modular forms}\label{HMF}
We recall results from \cite{Sh},\cite{H1} and \cite{H2}. Let $F$ be
a totally real field and denote by $\r$ its ring of integers.
Denote by $F_\mathbb A$ the adele ring of $F$ and by $\widehat F$ the ring
of finite adeles. For any place $v$ of $F$ and any $x\in F_\mathbb A$ or
$x\in\GL_2(F_\mathbb A)$, denote by $x_v$ the $v$-component of $x$. Fix
an open compact subgroup $U$ of $\GL_2(\widehat F)$. By the Strong
Approximation Theorem, choose $t_\lambda\in\GL_2(F_\mathbb A)$ for
$\lambda\in\{1,\dots, h(U)\}$ and a suitable integer
$h(U)$ depending on $U$ with $(t_\lambda)_{\fr q}=(t_\lambda)_v=1$  for any
prime ideal $\fr q$ dividing the adelized determinant $\widehat\det(U)$ of $U$
and any archimedean place $v$,
and such that there is the following
disjoint union decomposition:
\begin{equation}\label{disj}\GL_2(F_\mathbb
A)=\coprod_{\lambda=1}^{h(U)}\GL_2(F)t_\lambda
US_\infty\end{equation} where $S_\infty:=({\rm SO}_2(\R)\R_+)^d$ and
$\R_+$ is the group of positive real numbers.
Define the following arithmetic groups depending on ${U}$:
\begin{equation}\label{arith-groups}
\Gamma^\lambda(U ):=x_\lambda(U
)x_\lambda^{-1}\cap\GL_2^+(F) \mbox{ and }
\Gbar^\lambda(U ):=\Gamma^\lambda(U )/(\Gamma^\lambda(U )\cap
F^\times),\end{equation} where $\GL_2^+(F)$ is the subgroup of
$\GL_2(F)$ consisting of matrices with totally positive determinant.
For any $\lambda=1,\dots,h(U)$, denote by
$S_{k,w}(\Gamma^\lambda(U),\C)$ the $\C$-vector space of
Hilbert cusp forms with respect to the automorphic factor
$j(\gamma,z):=\frac{\det(\gamma)^w}{(cz+d)^{k}}$ for
$\gamma=\left(\begin{array}{cc}a&b\\c&d\end{array}\right)\in\Gamma^\lambda(U )$,
where the usual multi-index notations are used: for
$z=(z_\sigma)_{\sigma\in I}$ and $z'=(z'_\sigma)_{\sigma\in I}\in\fr
H[I]$ ($\fr H$ is the complex upper half plane) and
$m=(m_\sigma)_{\sigma\in I}\in\Z[I]$, $z^m:=\prod_{\sigma\in
I}z_\sigma^{m_\sigma}$ and $z+z'=(z_\sigma+z'_\sigma)_{\sigma\in
I}$. Explicitly, for any $\gamma\in\Gamma^\lambda(U )$:
$$(f|\gamma)(z):=\det(\gamma)^w(cz+d)^{-k}f(\gamma(z))=f(z),$$ where the
action of $\gamma$ on $\fr H[I]$ is given by composing the usual
action $z\mapsto\alpha(z)$ on $\fr H$ by fractional linear
transformations of the group $\GL_2^+(\R)$ of real matrices with
positive determinant with the injections
$\sigma:\GL_2^+(F)\hookrightarrow\GL_2^+(\R)$ deduced from
$\sigma\in I$; more precisely,
$\gamma((z_\sigma)_\sigma)=(\sigma(\gamma)(z_\sigma))_\sigma$.
Define
\[S_{k,w}({S}(U ),\C):=\prod_{\lambda=1}^{h(U )}S_{k,w}(\Gamma^\lambda(U ),\C).\]

For any ideal $\m \subseteq \r$ define the following open compact
semigroups of $\GL_2(\widehat F)$:
$$\Delta_0(\m):=\left\{\left(\begin{array}{cc}a&b\\c&d\end{array}\right)\in\GL_2(\widehat F)\cap\M_2(\r):
c\equiv 0\pmod{\fr m} \right\},$$
$$\Delta_1(\fr
m):=\left\{\left(\begin{array}{cc}a&b\\c&d\end{array}\right)\in\Delta_0(\fr
m):a_\q-1\in{\m }_\q \mbox{ for all prime ideals
}\q\mid\m\right\},$$ \begin{equation}\label{U0} U_0(\fr
m):=\left\{\gamma\in\Delta_0(\m
):\det(\gamma_\q)\in\widehat\r^\times_\q, \mbox{ for all prime
ideals }\fr q\subseteq \r\right\},\end{equation}
\begin{equation}\label{U1}
U_1(\fr m):=U_0(\m)\cap\Delta_1(\m ),\end{equation}
\begin{equation}\label{U}U(\fr
m):=\left\{\left(\begin{array}{cc}a&b\\c&d\end{array}\right)\in
U_1(\m ):d_\q-1\in{\m }_\q \mbox{ for all prime ideals
}\q\mid\m\right\}.\end{equation} where
$\widehat{\r}:=\r$ is the
profinite completion of $\r$.
Fix an ideal $\n \subseteq\r$, a rational prime $p$ prime to $2\n$
and not ramified in $F$ and an open compact subgroup
${S}\subseteq\GL_2(\widehat F)$ such that $U_0(\n )\supseteq
{S}\supseteq U_1(\n )$ and the $p$-component ${S}_p$ of ${S}$ is
$\GL_2(\r_p)$. For any positive integer $\alpha$, set
\[{S}_0(p^\alpha):={S}\cap U_0(p^\alpha),\quad {S}_1(p^\alpha):={S}\cap
U_1(p^\alpha),\quad {S}(p^\alpha):={S}\cap U(p^\alpha).\] Denote by
$\Gamma^\lambda(p^\alpha)$ and $\Gbar^\lambda(p^\alpha)$ (respectively,
$\Gamma^\lambda_0(p^\alpha)$ and $\Gbar^\lambda_0(p^\alpha)$;
$\Gamma^\lambda_1(p^\alpha)$ and $\Gbar^\lambda_1(p^\alpha)$) the arithmetic
groups associated as in \eqref{arith-groups}
to $S(p^\alpha)$ (respectively, $S_0(p^\alpha)$; $S_1(p^\alpha)$).


Any modular form $f_\lambda\in S_{k,w}(\Gamma^\lambda(p^\alpha),\C)$
has a Fourier expansion of the form
\[
f_\lambda(z)=\sum_{\xi\in\fr t_\lambda\fr d,\xi\gg
0}a_\lambda(\xi)e^{2\pi i(\xi\cdot z)}, \] where the notations are
as follows: $\fr t_\lambda$ is an ideal represented by
$t_\lambda$, $\fr d$ is the different ideal of $F/\Q$, $\xi\gg 0$
if and only if, by definition, $\xi$ is totally positive, and
$(\xi\cdot z):=\sum_{\sigma\in I}\sigma(\xi)z_\sigma$ is the
scalar product. For details, see \cite[Corollary 4.3]{H1}.

\subsection{Hecke operators}\label{section-Hecke-operators}

The right action  on $S_{k,w}({S}(p^\alpha),\C)$ of the Hecke
algebra $R({S}(p^\alpha),\Delta_0(\n p^\alpha))$, which is by
definition the free $\Z$-module generated by double cosets
\[T(x):={S}(p^\alpha)x{S}(p^\alpha)\] for $x\in\Delta_0(\n p^\alpha)$
with multiplication defined by \[\sum_ia_iT(x)\cdot\sum_jb_j
T(y):=\sum_{i,j}a_ib_jT(x_iy_j),\] can be described as follows. Fix
$x\in \Delta_0(\n  p^\alpha)$ and $\lambda\in\{1,\dots,h\}$. Let
$\mu\in\{1,\dots,h\}$ such that $\det(x)t_\lambda t_\mu^{-1}$ is
trivial in the strict class group of $F$. Let $\alpha_\lambda$ such
that ${S}(p^\alpha)x
{S}(p^\alpha)={S}(p^\alpha)x_\lambda^{-1}\alpha_\lambda x_\mu
{S}(p^\alpha)$ and form the finite disjoint coset decomposition
$\Gamma_\lambda\alpha_\lambda\Gamma_\mu=
\sum_{j}\Gamma_\lambda\alpha_{\lambda,j}$ where
$\alpha_{\lambda,j}\in\GL_2(F)\cap x_\lambda\Delta_0(\fr
np^\alpha)x_\mu^{-1}$. Define \[g_\mu:=\sum_jf_\lambda
|\alpha_{\lambda,j}\quad\mbox{ and } \quad
f|T(x):=(g_1,\dots,g_h).\]

Denote by $\widetilde F$ the composite field of all the images
$\sigma(F)$ of $F$ under the elements $\sigma\in I$ and by
$\widetilde \r$ its ring of integers. Fix an $\widetilde\r$-algebra
$A\subseteq\C$ such that for every $x\in\widehat F^\times$ and every
$\sigma\in I$, the $A$-ideal $x^\sigma A$ is generated by a single
element of $A$. For any prime ideal $\q \subseteq\r$, choose a
generator $\{\q ^{\sigma}\}\in A$ of the principal ideal
$\q^{\sigma} A$. Define $\{\q \}^v:=\prod_{\sigma\in
I}\{\q^\sigma\}^{v_\sigma}$. Write a fractional ideal $\m $ of $\r$
as a product of prime ideals $\fr m =\prod_{\q }\q ^{m(\q )}$ and
define $\{\m \}^v:=\prod_{\q }(\{\q \}^{v})^{m(\q )}$. For any
element $x\in\widehat F^\times$, denote by $\m_x$ the fractional
$\r$-ideal corresponding to $x$ and define $\{x\}^v:=\{\m_x\}^v$.
Modify the Hecke operators $T(x)\in R({S}(p^\alpha),\Delta_0(\n
p^\alpha))$  by setting:
\[T_0(x):=(\{x\}^v)^{-1}T(x).\] Denote by $h_{k,w}({S}(p^\alpha),A)$
the $A$-subalgebra of $\End(S_{k,w}({S}(p^\alpha),\C))$ generated
over $A$ by operators $T_0(x)$ for $x\in\Delta_0(\n p^\alpha)$. By
\cite[Proposition 1.1]{H2}, $h_{k,w}({S}(p^\alpha),A)$ is
commutative.

For $x\in F$ and $m=(m_\sigma)_\sigma\in\Z[I]$, set
$x^m:=\prod_{\sigma\in I}\sigma(x)^{m_\sigma}$. For any integral
ideal $\m$, choose $\lambda=\lambda(\fr m)$ such that $\fr m$ is
equivalent to $\fr t_\lambda\fr d$ in the strict ideal class group
of $F$ and let $\xi_\m \in\fr t_\lambda\fr d$ such that $\xi_\fr
m\gg 0$ and $\fr m=\xi_\fr m(\fr t_\lambda\fr d)^{-1}$. Define the
modified Fourier coefficients as in \cite[Corollary 3.4]{H1}:
\begin{equation}\label{fourier}
C(\fr m,f):=\frac{a_\lambda(\xi_\m )\xi_\m ^{v}}{b_{v,\lambda}},
\quad \mbox{ where } b_{v,\lambda}:=\frac{\norm{\fr
t_\lambda}}{\{(\fr t_\lambda\fr d)^{v}\}}.\end{equation} Suppose
that $f\in S_{k,w}({S}(p^\alpha),\C)$ is an eigenform for
$h({S}(p^\alpha),A)$ such that $C(\r,f_\lambda)=1$ for all
$\lambda=1,\dots, h$ (call such a form \emph{normalized}). Then by
\cite[Corollary 4.2]{H1}:
\[f|T(\m)=C(\fr m,f)f. \]

The group $G_\alpha:=
{S}_0(p^\alpha)\r^\times/{S}(p^\alpha)\r^\times$ acts on
$S_{k,w}({S}(p^\alpha),\C)$ via the operator
$\omega(a_p^{n+2v})^{-1}T(x)$ for
$x=\left(\begin{array}{cc}a&b\\c&d\end{array}\right)\in
{S}_0(p^\alpha)$, where $\omega$ is the Teichmuller character.

For any prime ideal $\q\subseteq\r$, choose an element
$q\in\widehat\r^\times$ such that $q\widehat\r=\widehat\q$ and the
$\fr l$-component $q_\fr l$ of $q$ is equal to 1 for all prime
ideals $\fr l\subseteq\r$, $\fr l\neq\q$. Define $T(\fr
q):=T\left(\begin{array}{cc}1&0\\0&q\end{array}\right)$ for all
prime ideals $\fr q\subseteq\r$ and $T(\fr q,\fr
q):=T\left(\begin{array}{cc}q&0\\0&q\end{array}\right)$ for prime
ideals $\fr q\subseteq \r$ such that $\fr q\nmid\n $. By
\cite[Proposition 1.1]{H2}, if ${S}\supseteq U_1(\fr n)$, then
$h_{k,w}({S}(p^\alpha),A)$ is generated over $A$ by the operators
induced from the action of $G_\alpha$ and $T_0(\q)$ for all prime
ideals $\q$.

If $f$ is an eigenform for the Hecke algebra
$h_{k,w}({S}(p^\alpha),A)$, then its eigenvalues are algebraic
numbers. If $k\sim t$, then they are algebraic integers.
Define $S_{k,w}({S}(p^\alpha),A)\subseteq S_{k,w}({S}(p^\alpha),\C)$
to be the $A$-module consisting of forms whose Fourier expansion has
coefficients in $A$. The $A$-module $S_{k,w}({S}(p^\alpha),A)$ is
stable under the action of $h_{k,w}({S}(p^\alpha),A)$.

\subsection{Nearly ordinary Hecke
algebras}\label{NOHA}

Choose an embedding $\iota:\overline\Q\hookrightarrow\overline\Q_p$,
so that any algebraic number is equipped with a $p$-adic valuation.
Fix a ring of integers $\cl O$ of a finite extension of the
completion of $\iota(\widetilde F)$. After choosing an embedding
$i:\overline\Q_p\hookrightarrow\C$, the $\widetilde\r$-algebra $\cl
O$ satisfies the conditions of \S \ref{section-Hecke-operators}.

By \cite[Lemma 2.2]{H2}, $h_{k,w}({S}(p^\alpha),\cl O)$ is free of
finite rank over $\cl O$ and hence can be decomposed as a direct
sum:
\begin{equation}\label{hecke-algebras}h_{k,w}({S}(p^\alpha),\cl
O)=h_{k,w}^\ord({S}(p^\alpha),\cl O)\oplus
h_{k,w}^{\rom{ss}}({S}(p^\alpha),\cl O)\end{equation} such that the
image of $T_0(p)$ in the first factor, the nearly ordinary part, is
is a unit while its image in the other factor is topologically
nilpotent. For any pair of non negative integers $\beta\geq\alpha$
the map $T_0(x)\mapsto T_0(x)$ for $x\in\Delta_0(\n p^\alpha)$
induces a surjective ring homomorphism
$\rho_{\alpha}^\beta:h_{k,w}({S}(p^\beta),\cl O)\to
h_{k,w}({S}(p^\alpha),\cl O)$. Define: $$h_{k,w}({S}(p^\infty),\cl
O):=\invlim h_{k,w}({S}(p^\alpha),\cl O)$$ and  $$h_{k,w}^\ord
({S}(p^\infty),\cl O):=\invlim h_{k,w}^\ord({S}(p^\alpha),\cl O)$$
where the inverse limits are with respect to the maps
$\rho_\alpha^\beta$. By \cite[Theorem 2.3]{H2}, for any weight
$(k,w)$ there is an isomorphism $h_{k,w}({S}(p^\infty),\cl O)\simeq
h_{2t,t}({S}(p^\infty),\cl O)$ which takes $T(\q)$ to $T(\q)$ and
$T(\q,\q)$ to $T(\q,\q)$ for all prime ideals $\q\nmid p$. This
isomorphism induces an isomorphism between the nearly ordinary parts
\begin{equation}\label{def-R}h_{k,w}^\ord({S}(p^\infty),\cl O)\simeq \cl
R:=h_{2t,t}^\ord({S}(p^\infty),\cl O).\end{equation}

Set $S_F:=S\cap\widehat F^\times$, $S_F(p^\alpha):=S(p^\alpha)\cap
\widehat F^\times$, $\overline Z_\alpha:=
S_F\r^\times/S_F(p^\alpha)\r^\times$ and $\overline
Z_\infty:=\invlim \Zbar_\alpha$. By \cite[Lemma 2.1]{H2} the map
$\mat abcd \mapsto (a_p^{-1}d_p,a)$ induces an isomorphism
$G_\alpha\simeq(\r/p^\alpha)^\times\times\Zbar_\alpha$ and hence an
isomorphism $$G:=\invlim
G_\alpha\simeq\r_p^\times\times\Zbar_\infty.$$ Since
$h_{k,w}({S}(p^\alpha),\cl O)$ is a $\cl O[G_\alpha]$-algebra and
this algebra structure is compatible with the maps
$\rho_\alpha^\beta$, it follows that $h^\ord_{k,w}(S(p^\infty),\cl
O)$ is a $\widetilde{\Lambda}:=\cl O[\![G]\!]$-algebra. Write
$G\simeq W\times G^\rom{tors}$ where $G^\rom{tors}$ is the torsion
subgroup of $G$ and $W$ is torsion-free. Note that $W$ is well
determined only up to isomorphism; fix from now on a choice of $W$.
the ring $\Lambda:=\cl O[\![W]\!]$ is isomorphic to the power series
ring $\cl O[\![X_1,\dots,X_s]\!]$ in $d<s<2d$ variables, where
$d:=[F:\Q]$. If Leopold's conjecture holds, then $s=d+1$. By
\cite[Theorem 2.4]{H2}, the nearly ordinary Hecke algebra $\cl
R=h_{2t,t}^\ord({S}(p^\infty),\cl O)$ is a torsion-free
$\Lambda$-module of finite type.

For any finite order character
$\epsilon:{S}_0(p^\alpha)/{S}_1(p^\alpha)\to\overline\Q^\times,$
define characters
$\epsilon_\lambda:\Gamma_0^\lambda(p^\alpha)/\Gamma_1^\lambda(p^\alpha)\to\C^\times$
by setting $\epsilon_\lambda(\gamma):=\epsilon(x_\lambda^{-1}\gamma
x_\lambda)$. Denote by $S_{k,w}({S}_0(p^\alpha),\epsilon,\C)$ the
$\C$-subspace of $S_{k,w}({S}(p^\alpha),\C)$ consisting of forms
$f=(f_1,\dots,f_h)$ such that for any $\lambda=1,\dots,h$ and any
$\gamma\in\Gamma_0^\lambda(p^\alpha)$,
$(f_{\lambda}|\gamma)(z)=\epsilon^{-1}_\lambda(\gamma)f(z).$ Suppose
that $\cl O$ contains the values of $\epsilon$ and set
$S_{k,w}(S_0(p^\alpha),\epsilon,\cl O):=S_{k,w}(S_0(p^\alpha),\epsilon,\C)\cap
S_{k,w}(S(p^\alpha),\cl O)$.
Denote by
\[h_{k,w}(S_0(p^\alpha),\epsilon,\cl O)\] the $\cl O$-subalgebra of
$\End(S_{k,w}(S_0(p^\alpha),\epsilon,\C))$ generated over $\cl O$ by
operators $T_0(x)$ for $x\in\Delta_0(\n p^\alpha)$. Define
$h_{k,w}^\ord(S_0(p^\alpha),\epsilon,\cl O)$ to be the maximal
factor of $h_{k,w}(S_0(p^\alpha),\epsilon,\cl O)$ such that the
image of $T_0(p)$ is a unit in that factor. (In the following the
$\epsilon_\lambda$'s will often be simply denoted by $\epsilon$).

\begin{defi}\label{def-ordinary}
An eigenform $f\in S_{k,w}({S}_0(p^\alpha),\epsilon,\C)$ for the Hecke algebra
$h_{k,w}({S}(p^\alpha),A)$ is said to be \emph{$p$-nearly ordinary}
if the eigenvalue of $T_0(p)$ is a $p$-adic unit.
\end{defi}

Denote by $\r^\times_+$ the group of totally positive units of $\r$.
Define \begin{equation}\label{Z}Z_\alpha:=S_F\r^\times_+/S_F(p^\alpha)\r^\times_+\end{equation} and
$Z_\infty:=\invlim Z_\alpha$. The kernel of the natural surjection
map $Z_\infty\to\Zbar_\infty$ is finite and annihilated by a power
of 2. Denote by $\chi_\cyc:Z_\infty\to\Z_p^\times$ the cyclotomic
character defined by $\chi_\cyc(x)=x^t=\prod_{\sigma\in
I}\sigma(x)=\norm{x}$. Let $\epsilon: Z_\infty\to\cl O^\times$ be a
character factoring through $Z_\alpha$. Suppose that
$\epsilon\chi_\cyc^{n+2v}$ factors through $\Zbar_\infty$, where if
$n+2v=mt$ with $m\in\Z$, then $\chi_\cyc^{n+2v}$ is by definition
$\chi_\cyc^m$. Let
$$P_{n,v,\epsilon}:G\simeq\r_p^\times\times\Zbar_\infty\to\cl
O^\times$$ be the character defined by
$P_{n,v,\epsilon}(a,z):=\epsilon\chi_\cyc^{n+2v}(z)a^v$. Denote by
the same symbol the homomorphism $\widetilde{\Lambda}\to\cl O$
deduced from $P_{n,v,\epsilon}$ by extension of scalars. The kernel
of this homomorphism is a prime ideal of $\widetilde{\Lambda}$,
denoted be the same symbol $P_{n,v,\epsilon}$. To simplify
notations, set $P:=P_{n,v,\epsilon}$. Let $\widetilde\Lambda_{P}$
(respectively, $\cl R_{P}$) denote the localization of
$\widetilde\Lambda$ (respectively, $\cl R$) in $P$. Then $\cl R_{P}$
is free of finite rank over $\widetilde\Lambda_{P}$ and the natural
surjective morphism $\cl R\to h_{k,w}(S_0(p^\alpha),\epsilon,\cl O)$
induces by \cite[Theorem 2.4]{H2} an isomorphism:
\begin{equation}\label{control-theorem-hecke-algebras}\cl R_{P}/P\cl R_P\simeq
h_{k,w}(S_0(p^\alpha),\epsilon,K),\end{equation} where
$K:=\rom{Frac}(\cl O)$ is the fraction field of $\cl O$. In
particular, the dimension of the $K$-vector space
$h_{k,w}(S_0(p^\alpha),\epsilon,K)$ does not depend on $\epsilon$
and $(k,w)$ and is equal to the $\Lambda_P$-rank of $\cl R_P$.

Let $\cl L:=\rom{Frac}(\Lambda)$ be the fraction field of $\Lambda$
and fix an algebraic closure $\overline{\cl L}$ of $\cl L$. Let
$\theta:\cl R\to\overline{\cl L}$ be a $\Lambda$-algebra
homomorphism. The image $\Image(\theta)$ of $\theta$ is finite over
$\Lambda$. Denote by $\cl I$ the integral closure of
$\Image(\theta)$ in the fraction field $\cl
K:=\rom{Frac}(\Image(\theta))$. Define
$$\cl X(\cl I):=\Hom_{\cl O\rom{-alg}}(\cl I,\overline\Q_p)$$ and
denote by $\cl A(\cl I)$ the subset of $\kappa\in\cl X(\cl I)$
consisting of points whose restriction to $\Lambda$ coincide with
the restriction to $\Lambda$ of some character
$P_{n(\kappa),v(\kappa),\epsilon(\kappa)}$ as above. Points in $\cl
A(\cl I)$ are called \emph{arithmetic}. In this case, denote
$P_{n(\kappa),v(\kappa),\epsilon(\kappa)}$ by $P_\kappa$ and set
$k(\kappa):=n(\kappa)-2t$ and $w(\kappa):=k(\kappa)+v(\kappa)-t$.
Let $C(\kappa)$ denote the conductor of $\epsilon$ restricted to the
torsion free part $W$ of $\Zbar_\infty$ and $\epsilon_W$ the
restriction of $\epsilon$ to $W$. Let $\Zbar_\infty^\rom{tors}$
denote the maximal torsion subgroup of $\Zbar_\infty$ and let
$\psi:\Zbar_\infty^{\rom{tors}}\to\overline{\cl L}$ be the composite
of $\theta$ with the natural map $\Zbar_\infty^{\rom{tors}}\to\cl
R^\times$ induced by the action of $G$ on $\cl R$. Define
$\r_p^{\times \rom{tors}}$ denote the maximal torsion subgroup of
$\r_p^\times$. For any $\kappa\in\cl X(\cl I)$, define
$$\theta_\kappa:=\kappa\circ\theta:\cl R\to\overline\Q_p.$$ By \cite[Corollary
2.5]{H2}, if $\kappa\in\cl A(\cl I)$ and $\theta$ restricted to
$\r_p^{\times\rom{tors}}$ is the character $x\mapsto x^{v(\kappa)}$,
then $\theta_\kappa(T(\q))$ are algebraic numbers for all prime
ideals $\q$ and there exists a unique up to constant factors
$p$-nearly ordinary eigenform $$f_\kappa\in
S_{k(\kappa),w(\kappa)}(U_0(\fr
nC(\kappa)),\epsilon_W\psi\omega^{-(n(\kappa)+2v(\kappa))},\C)$$
such that $f_\kappa|T(\q)=\theta_\kappa(T(\q))f_\kappa$, where
$\omega$ is the Teichmuller character and if
$n(\kappa)+2v(\kappa)=mt$ with $m\in\Z$, then
$\omega^{-(n(\kappa)+2v(\kappa))}:=\omega^{-m}$. Conversely, if
$\alpha>0$ and $f\in S_{k,w}(U_1(\n p^\alpha),\C)$ is a $p$-nearly
ordinary eigenform, then there exists $\kappa\in\cl A(\cl I)$ and
$\theta$ as above such that $f$ is a constant multiple of
$f_\kappa$.

\section{$\Lambda$-adic  modular symbols}

\subsection{Classical modular symbols}\label{CMS}
Define the Hilbert variety associated to an open compact
 subgroup $U$ of $\GL_2(\widehat F)$ to be the complex variety:
\[X_{U}:=\GL_2(F)\backslash\GL_2(F_\mathbb
A)/U\cdot S_\infty.\] By strong
approximation, \[X_{U}\simeq
\coprod_{\lambda=1}^{h(U)}\Gamma^\lambda(U)\bk\fr H^d.\]
Suppose that $E$ is a (right or left)
$\Gamma^\lambda(U )$-module for all $\lambda$. Denote by $\cl
E$ the coefficient system on $X_{U }$ associated to $E$.
Then
$$H^d (X_{U },\cl E)\simeq
\oplus_{\lambda=1}^{h(U)} H^d (\Gamma^\lambda(U )\bk\fr
H^d,\cl E)\simeq \oplus_{\lambda=1}^{h(U)} H^d
(\Gbar^\lambda(U ),E).$$ For any $\omega\in H^d
(X_{U },\cl E)$, write $\omega_\lambda$ for its projection
to $H^d (\Gamma^\lambda(U )\bk\fr H^d,\cl E)$.

\begin{defi}
The group of \emph{modular symbols} on $X(U )$ associated to
$E$ is the group $H^d_\cpt (X_{U },\cl E)$ of cohomology
with compact support.\end{defi}

Suppose that $E$ is a right $t_\lambda\Delta_0(\n
p^\alpha)t_\lambda^{-1}\cap\GL_2(F)$ for all $\lambda$. Define an
action of the Hecke algebra $R(S(p^\alpha),\Delta_0(\n p^\alpha ))$
by the formula:
$$\omega_\lambda(z)|T(x):=\sum_{j}\omega_\lambda(\alpha_{\lambda,j}z)|\alpha_{\lambda,j}
\in H^d (\Gamma^\mu(p^\alpha )\bk\fr H^d,\cl E)$$ (same notations as
in \S \ref{section-Hecke-operators}). Equivalently, identifying
$\omega_\lambda$ with a $d$-cocycle by the above isomorphism, $T(x)$
can be defined as in \cite{as} by the formula:
$$(\omega_\lambda)|T(x)(\gamma_0,\dots,\gamma_d):=\sum_j
\omega_\lambda(t_j(\gamma_0),\dots,t_j(\gamma_d))|\alpha_{\lambda,j},$$
where $t_j:\Gamma^\mu(p^\alpha)\to\Gamma^\lambda(p^\alpha )$ are
defined for $\gamma\in\Gamma^\mu(p^\alpha)$ by the equations
$\Gamma_\lambda \alpha_{\lambda,j}\gamma=\Gamma_\lambda
\alpha_{\lambda,l}$ and
$\alpha_{\lambda,j}\gamma=t_j(\gamma)\alpha_{\lambda,l}$.

The group of modular symbols $H^d_\cpt (X_{S(p^\alpha)},\cl E)$ is
an $R({S}(p^\alpha),\Delta_0(\n p^\alpha))$-module if $E$ is.

The modular symbol $\omega(f)$ associated to $f\in
S_{k,w}(S(p^\alpha),\C)$ can be described as follows. For any ring
$R$, let $L(n,R)$ be the $R$-module of homogeneous polynomials in
$2d$ variables $X=(X_\sigma)_{\sigma\in I},$ and
$Y=(Y_\sigma)_{\sigma\in I}$ of degree $n_\sigma$ in
$X_\sigma,Y_\sigma$. Denote by $\gamma=\mat abcd
\mapsto\gamma^*:=\mat d{-b}{-c}a$ the main involution of $\M_2(R)$.
Define a right action of $\GL_2(F)$ on $L(n,\C)$ by
$(P|\gamma)(X,Y):=\det(\gamma)^vP((X,Y)\gamma^*)$, where
$(X,Y)\gamma^*$ is matrix multiplication. Denote by $L(n,v,\C)$ the
right representation of $\GL_2(F)$ thus obtained. The differential
form $\omega(f_\lambda)(z):=f(z)(zX+Y)^{k-2}dz$ (usual multi-index
notations) satisfies the transformation formula for any
$\gamma\in\Gamma^\lambda(p^\alpha)$:
$$(\omega(f_\lambda))(\gamma(z))|\gamma =
\omega(f_\lambda)(z)
$$ hence, by \cite{MS}, $\omega(f_\lambda)\in
H^d(\Gamma^\lambda(p^\alpha)\bk\fr H^d ,\cl L(n,v,\C))$, where $\cl
L(n,v,\C)$ is the coefficient system on
$\Gamma^\lambda(p^\alpha)\bk\fr H^d$ associated to $L(n,v,\C)$.
Since $f_\lambda$ is a cusp form, it can be proved that
$\omega(f_\lambda)$ has compact support. Hence,
$\omega(f):=(\omega(f_1),\dots,\omega(f_h))\in
H^d_\cpt(X_{S(p^\alpha)},\cl L(n,v,\C))$.

For any character $\epsilon$ as above, write $L(n,v,\epsilon,R)$ for
the $\Delta_0(\fr n)$-module $L(n,v,R)$ with the action of
$\Delta_0(\fr n)$ twisted by $\epsilon$, that is, denoting by
$|_\epsilon$ this new action:
$P|_\epsilon\gamma:=\epsilon(\gamma)P|\gamma$ for
$\gamma\in\Delta_0(\fr n)$. If $f_\lambda\in
S_{k,w}(S_0(p^\alpha),\epsilon,\C)$, then $\omega(f_\lambda)\in
H^d(\Gamma_0^\lambda(p^\alpha),\cl L(n,v,\epsilon,\C))$, where $\cl
L(n,v,\epsilon,\C)$ is the coefficient system associated to
$L(n,v,\epsilon,\C)$.  Hence $\omega(f)\in
H^d_\cpt(X_{S_0(p^\alpha)},\cl L(n,v,\epsilon,\C))$.

A straightforward calculation shows that the map $f\mapsto\omega(f)$
is equivariant for the action of $R(S(p^\alpha),\Delta_0(\n
p^\alpha))$.

\subsection{$\Lambda$-adic  modular symbols}\label{OMS}

Define $\e=\e_S:=S\cap F^\times$. Then
$\Zbar_\alpha=(\r_p/p^\alpha\r_p)^\times/\e$, so
$G=\r_p^\times\times\r_p^\times/{\ec},$ where ${\ec}$ is the closure
of $\e$ in $\r_p^\times$. It follow that $G\simeq
(\r_p^\times\times\r_p^\times)/{\ec}$ via the map
$(x,y)\mapsto(xy,y)$. Embed diagonally ${\ec} $ in $\GL_2(\r_p)$ and
call $C$ the image. Let $N$ be the standard lower unipotent subgroup
of $\GL_2(\r_p)$. Define:
\begin{equation}\label{X}X:=NC\bk \GL_2(\r_p)\simeq\invlim
X_\alpha, \mbox{ where } X_\alpha:=S(p^\alpha)\r^\times\bk
S\r^\times.\end{equation} Let $F_p = F \otimes_\Q \Q_p =
\prod_{\frakp|p} F_\frakp$ and $Y:= N(F_p)C \bk \GL_2(F_p)$, where
$N(F_p)$ is the group of lower triangular matrices with entries in
$F_p$ and diagonal entries $1$. Write $\r_p=\prod_{\fr p\mid
p}\r_\fr p$, where $\r_\fr p$ is the completion of $\r$ at $\fr p$.
Define $(\r_\fr p^2)^\prime$ to be the set of primitive vectors of
$\r_\fr p^2$, that is, the pair of elements $(a,b)\in \r_p^2$ such
that at least one of $a$ and $b$ does not belong to $\fr p$. Set
$(\r_p^2)^\prime:=\prod_{\fr p\mid p}(\r_\fr p^2)^\prime.$ The map
$g=\mat abcd \mapsto ((a,b),\det(g))$ defines a bijection between
$X$ and ${\ec} \bk ((\r_p^2)^\prime \times \r_p^\times)$, where the
action of $e\in{\ec} $ is $e\cdot((x,y),z)=((ex,ey),e^2 z)$.

Define $\pi_\frakp$ to be the element in $\GL_2(\widehat F)$ whose
$\frakp$-component is $\mat 100{\frakp}$ and is 1 outside $\frakp$.
Note that $\pi_\fr p$ normalizes $N(F_p)=\prod_{\fr p^\prime\mid
p}N(F_{\fr p^\prime})$ because $\mat 100{\fr p}$ normalizes $N(F_\fr
p)$. Hence, it is possible to define an action of $\pi_\frakp$ on
$Y$ by letting $\pi_\fr p$ act on its $\fr p$-component as:
$$N(F_\frakp)g * \pi_\frakp = N(F_\frakp)\pi_\frakp^{-1} g \pi_\frakp.$$
Identify $Y$ with $\ec\bk((F_p^2)^\prime\times F_p^\times)$, where
$(F_p^2)^\prime:=\{(x,y)\in F_p^2: xy\neq 0\}$ via the map
$\gamma=\mat abcd\mapsto((a,b),\det(\gamma))$. Then
\[((x,y),z)*\pi_\frakp=((x,\frakp y),z),\] where for any
$y=\prod_{\fr p^\prime\mid p}y_\fr p$, write $\fr py:=\prod_{\fr
p^\prime\mid p, \fr p^\prime\neq\fr p} y_{\fr p^\prime}\times \fr
py_\fr p$. In particular, $\pi_\frakp$ does not affect the
determinant of the matrix.

Let $G^\pi$ be the semigroup generated by $\GL_2(\r_p)$ and
$\pi_\frakp$ for all divisors $\fr p$ of $p$. Using that any element
$s\in G^\pi$ can be expressed as a word in terms of $\GL_2(\r_p)$
and $\pi_\frakp$, and that the actions of the $\pi_\fr p$'s commute,
extend the $*$ action to $G^\pi$ by letting any $\pi_\fr p$ act
through $*$ and elements of $\GL_2(\r_p)$ through right
multiplication, so that \[N(F_p)g * s = N(F_p)\prod_{\fr p\mid
p}\pi_\fr p^{-c(\fr p)} g s\] for any $N(F_p)g \in Y$ and $s \in
G^\pi$, where $c(\fr p)$ is the number of times $\pi_\fr p$ appears
in the expression of $s$. Since this number does not depend on the
specific expression we chose, the action of $G^\pi$ on $Y$ is well
defined.


Let $Y^\prime$ denote the smallest subset of $Y$ containing $X$ and
stable under $G^\pi$. Define $\D_X$ (respectively, $\D_{Y^\prime}$)
to be the $\cl O$-module of $\cl O$-values measures on $X$
(respectively, on $Y^\prime$). For $s\in G^\pi$ and
$\mu\in\D_{Y^\prime}$, define $\mu*s$ by the integration formula:
$$\int_{Y^\prime}\varphi({\eta})d(\mu*\gamma)({\eta}):=\int_{Y^\prime}\varphi({\eta}*\gamma)d\mu({\eta}),$$
where $\varphi$ is any $\C_p$-valued step function on $Y^\prime$.
Denote by $\D_X\xrightarrow{i} \D_{Y^\prime}$ the canonical
inclusion defined by extending measures by zero and by
$\D_{Y^\prime}\xrightarrow{p} \D_X$ the canonical projection map.
If $\mu \in \D$ and $s \in G^\pi$, define $$\mu*s = p(i(\mu)*s).$$
Since by \cite[Lemma 3.1]{as} the kernel of $p$ is stable under
$G^\pi$, the action is well defined.

By the choice of $t_\lambda$ in \S \ref{section-2},
$\Gamma^\lambda(p^\alpha)\subseteq\GL_2(\r_p)$ for all $\lambda$.
Hence, ${\D}_X$ is a right $\Gamma^\lambda(p^\alpha)$-module for all
$\lambda$. Denote by $\cl D_X$ the coefficient system associated to
${\D}_X$ and set: $$\W:=H^d_\cpt(X_S,\cl D_X).$$


\begin{rem} Since $\rom{SO}_2(\R)$ is compact and isomorphic to the unit
circle $\C^1$ and the $\Gamma^\lambda$ are discrete, the stabilizer
$(\Gbar^\lambda)_{z_0}$ of any element $z_0\in\fr H^d$ is a finite
cyclic group. Since the groups $\Gbar^\lambda$ are torsion-free, it
follows that $\gamma$ is in the center of $\Gamma^\lambda$ and hence
acts trivially on $\D_X$. Hence the sheaf $\scrd_X$ is well-defined.
\end{rem}

Since $t_\lambda\Delta_0(\fr n)t_\lambda^{-1}\subseteq G^\pi$ for
all $\lambda=1,\dots,h$, it follows that $\W$ is an $R_\cl
O({S},\Delta_0(\n)):=\cl O\otimes_\Z R({S},\Delta_0(\n))$-module.
Let $G$ act on $X$ by left multiplication. Define $G'$ to be the
multiplicative subset of $(\r\times\r)/\r^\times$ consisting of pair
of elements $(x,y)$ such that $x$ and $y$ are prime to $p$. The map
$(a,d)\mapsto\omega(a_p^{n+2v})^{-1}T \mat a00d$ for $(a,d)\in G'$
considered in \S \ref{section-Hecke-operators} is multiplicative,
hence extends to a $\cl O$-algebra homomorphism $\cl O[G']\to R_\cl
O(S,\Delta_0(\fr n))$. On the other hand, $G'\subseteq G$, hence
$\cl O[G']$ embeds naturally on $\widetilde{\Lambda}=\cl
O[\![G]\!]$. Form the $\widetilde{\Lambda}$-algebra
$$\cl H:=R_\cl O({S},\Delta_0(\n))\otimes_{\cl
O[G']}\widetilde{\Lambda}.$$ Since the action of $G'$ on $\W$
extends to a continuous action of $G$, it follows that $\W$ is an
$\cl H$-module.

From the fact that $h_{2t,t}(S(p^\alpha),\cl O)$ is generated over
$\cl O$ by $T(\q)$ for all prime ideals $\q$ and those operators
coming from the action of $G_\alpha$, it follows that there is are
surjective homomorphisms of $\widetilde{\Lambda}$-algebras $\cl H\to
h_{2t,t}(S(p^\infty),\cl O)$ and $\cl H\to\cl R$.

Define a subset $X^\prime$ of $X$ as follows:
\begin{equation}\label{X'}X^\prime = \left\{\left. x = NC \mat abcd \in X
\right| a \in \r_p^\times \right\}.\end{equation} It is easy to
check that the definition does not depend on the choice of the
representative matrix used to define it and that $X^\prime$ can be
identified with the set $\ec\bk(\r_p^\times \times \r_p \times
\r_p^\times)$ under the above identification between $X$ and
$\ec\bk((\r_p^2)^\prime\times\r_p^\times)$. From now on denote
elements of $X$ by $((x,y),z)$, where $(x,y)$ is the first arrow on
the matrix and $z$ is its determinant.

Let $\kappa\in\cl A(\cl I)$ be an arithmetic point of weight $(n,v)$
and character $\epsilon$ factoring through $Z_\alpha$. Define the
\emph{specialization map} $\rho_{\kappa}:{\D_X}\to L
(n,v,\epsilon,\cl O)$ at $P$ by:
$$\mu \mapsto \rho_{\kappa}(\mu):=\int_{X^\prime}z^v\epsilon (x)(xY -
yX)^{n}d\mu(x,y,z).$$ Suppose that the conductor of $\epsilon$ is
$p^\alpha$ for some non negative integer $\alpha$. A simple
computation shows that
\[\rho_{\kappa}(\mu*\gamma)=\rho_{\kappa}(\mu)|_\epsilon\gamma\]
for $\gamma\in\GL_2(\r_p)\cap\Delta_0(p^\alpha)$. It follows that
the specialization map $\rho_{\kappa}$ is $\GL_2(\fr
r_p)\cap\Delta_0(p^\alpha)$-equivariant. Letting $K:=\rom{Frac}(\cl
O)$, there are $\GL_2(\r_p)\cap\Delta_0(p^\alpha)$-equivariant maps:
$$\rho_{\kappa}:\W\to\W_{\kappa}:=H^d_\cpt(X_{S(p^\alpha)},\cl L(n,v,\epsilon,\cl
O)).$$

\begin{proposition}\label{hecke/actions} Let $\Phi\in\W$.
\begin{enumerate}
\item For any prime ideal $\q$ of $\r$ prime to $p$:
$\rho_{\kappa}(\Phi*T(\fr q))=(\rho_{\kappa}(\Phi))|T(\fr q)$.
\item\label{part-2}
$\rho_{\kappa}(\Phi*T(p))=(\rho_{\kappa}(\Phi))|T_0(p)$.
\end{enumerate}\end{proposition}

\begin{proof}
The equivariance for the action of Hecke operators $T(\q)$ in the
first statement is immediate because of the
$\GL_2(\r_p)\cap\Delta_0(p^\alpha)$-equivariance of $\rho_{\kappa}$.
It remains to check the action of $T(p)$. To see this, write
$\Gamma^\lambda\pi\Gamma^\lambda=\coprod_{t}\Gamma^\lambda\alpha_{\lambda,t}$
and note that
\begin{eqnarray} \rho_{\kappa} (\Phi*T_0(p) ) &=& \int_{X^\prime} \sum_t
z^v \epsilon(x)(xY-yX)^n\ d(\Phi*\alpha_{\lambda,t})\nonumber \\
&=&
\int_{X^\prime} \sum_t z^v \epsilon(x)(xY-yX)^n\ d(\pi^{-1} \Phi \alpha_{\lambda,t})\nonumber \\
&=& \sum_t \int_{X^\prime}
\{p^v\}^{-1} z^v \epsilon(x)(xY-yX)^n|\alpha_{\lambda,t}\ d(\Phi) \nonumber \\
&=& \rho_{\kappa} (\Phi)|T_0(p) . \nonumber \end{eqnarray} This
proves the second formula.\end{proof}

Note that
\[\W=\prod_{\lambda=1}^h\W^\lambda, \mbox{ where }\W^\lambda=H^d_{\rm
cpt}(\Gbar^\lambda\bk\fr H^d,\cl D_X)\]
\[\W_\kappa=\prod_{\lambda=1}^h\W^{\lambda}_\kappa, \mbox{ where }
\W^{\lambda}_\kappa:=H^d_{\rm
cpt}(\Gbar^{\lambda}_0(p^\alpha)\bk\fr H^d,\cl L(n,v,\epsilon,\cl
O)).\] Any element
$\Phi\in\W$ will be written as $(\Phi_\lambda)_{\lambda=1,\dots,h}$
while any element of $\omega\in\W_\kappa$ will be denoted as
$(\omega_{\lambda})_{\lambda=1,\dots,h}$. Define:
\begin{equation}\label{Phi-theta-kappa}
\Phi_{\lambda,\theta_\kappa}:=\rho_\kappa(\Phi_\lambda)\in
\W_{\kappa}^{\lambda}.\end{equation}

\subsection{The Control Theorem}\label{sec-control-th}

Fix $\theta:\cl R\to\overline{\cl L}$ (where $\cl
L=\rom{Frac}(\Lambda)$) and denote as in \S \ref{NOHA} by $\cl I$
the integral closure of $\Lambda$ in $\Image(\theta)$. Recall the
specialization map $\theta_\kappa:=\kappa\circ\theta : \scrr \to
\qpbar$ which corresponds to an eigenform $f_\kappa$. The map
$\theta_\kappa$ extends to the localization $\scrr_{P_\kappa}$ of
$\cl R$ at $P_\kappa$ and one can intertwine $\theta_\kappa$ with
the map $\rho_{\kappa}$ defining
\begin{equation}\label{rho-star} \rho_\kappa : \W \times
\scrr_{P_\kappa} \rightarrow \W_{\kappa}\end{equation} by
$$\sum_\lambda \Phi_\lambda \times r_\lambda \mapsto \sum_\lambda
\rho_\kappa(\Phi_\lambda) \cdot \theta_{\kappa} (r_\lambda).$$ This
is well defined because, if $g$ belongs to the free part $W$ of $G$
and is represented by the matrix $\mat a00d \in \GL_2(\r_p)$, then
we have for $\mu \in \D_X$,
$$ \rho_{\kappa} (g \mu)=\rho_{\kappa}(g) \rho_{\kappa} (\mu).$$
Since $\rho_\kappa (\phi g, r) = \rho_\kappa (\phi, gr)$ for $g \in
W$ and by continuity the same is true for any element in $\Lambda$,
the map \eqref{rho-star} induces a homomorphism
$$\rho_\kappa : \W \otimes_{\Lambda} \scrr_{P_\kappa} \rightarrow
\W_{\kappa}$$ which is Hecke equivariant.

For any $\cl H$-module $M$, let $M^\o$ denote its ordinary part,
that is, the maximal subspace of $M$ on which the $T(p)$ operator
acts as a unit. Let $h: \scrh \to \scrr$ be the natural map obtained
by the action of Hecke operators on $\Lambda$-adic  cusp forms. For
any arithmetic point $\kappa$, let $h_\kappa$ be the composition of
$h$ with the localization morphism $\cl R\to \cl R_{P_\kappa}$. For
any $\scrh \otimes_\Lambda \scrr_{P}$-module $M$, let
$$M^{h_\kappa}
=\{ m \in M \ |\ (T(\q) \otimes 1) m = (1 \otimes h_\kappa(T(\q)))
\cdot m \textrm{ for all prime ideals } \q \textrm{ in } \fr r \}$$
denote the $h_\kappa$-eigenspace of $M$. If $f_\kappa$ is a
classical eigenform for an arithmetic point $\kappa$, let
$$\W^{f_\kappa}_{\kappa} = \{ \phi \in \W_{\kappa} |\
T_0(\fr q)\phi = a_\fr q(g) \phi \}$$ denote the
$f_\kappa$-eigenspace of $\W_{\kappa}$, where $a_\fr q(g)$ is the
eigenvalue of the $T_0(\fr q)$ operator on $f_\kappa$. Hence, there
is a map:
$$\rho_\kappa: (\W \otimes_{\Lambda} \scrr_{P_\kappa})^{h_\kappa} \to
\W_{\kappa}^{f_\kappa}.$$ The action of the involution
\[\tau:=\left( \mat {1}00{-1},\dots,\mat {1}00{-1} \right) \in
\GL_2(\r_p)\cap\Delta_0(p^\alpha)\] on
$\GL_2(\r_p)\cap\Delta_0(p^\alpha)$-modules gives rise to $2^d$
eigenspaces indexed by $\sgn \in \{ \pm \}^d$. For each
$\GL_2(\r_p)$-module $M$, let $M^\sgn$ denote the corresponding
eigenspace.

Note that, since $S\supseteq U_1(\fr n)$, there is a notion of
\emph{primitive homomorphisms} which can be introduced as in
\cite[pages 317, 318]{H1}. Say that an arithmetic point $\kappa$ is
\emph{primitive} if both $\theta$ and $\theta_\kappa$ are primitive
characters. Note that in particular $f_\kappa$ is a primitive form
of level $U_1(\fr n)$ (see \cite[(3.10b)]{H1}) and that, by
\cite[Corollary 3.7]{H1},
\[\cl R\otimes_\Lambda \cl K\simeq \cl K\oplus \cl B\] as an algebra
direct sum such that the projection to $\cl K={\rm Frac}({\rm
Im}(\theta))$ coincides with $\theta$ on $\cl R$.

\begin{theorem}\label{main-theorem} Let $\kappa\in\cl A(\cl I)$ be
a primitive arithmetic point of weight $(n_\kappa ,v_\kappa )$ and
character $\epsilon_\kappa $. For each $\sgn \in \{ \pm \}^d$ the
map
$$\rho_\kappa: (\W \otimes_{\Lambda}
\scrr_{P_\kappa })^{h_\kappa ,\sgn}/P_\kappa  (\W \otimes_{\Lambda}
\scrr_{P_\kappa })^{h_\kappa ,\sgn}\rightarrow\W_{\kappa }^{f_\kappa
, \sgn}$$ is an isomorphism.\end{theorem} The proof of this Theorem
will be given in \S \ref{proof-control-theorem}. Before of
explaining the proof, we need some preliminary results, stated in \S
\ref{section-prop-1} and \S \ref{section-prop-2}.

\subsection{Description of $\Ker(\rho_\kappa)$}\label{section-prop-1}

\begin{proposition}\label{first-prop} The group
$\W^\o$ of {ordinary $\Lambda$-adic modular symbols} is a free
$\Lambda$-module of finite rank. The kernel of $\rho_\kappa$ is
equal to $P_\kappa \W^\o$.
\end{proposition}

\begin{proof} This is \cite[Theorem 5.1]{as}, so only a sketch of
the proof will be given. First prove the equality $P_\kappa \W^{\o
}=\Ker(\rho_\kappa )$.

\begin{enumerate}
\item $\Ker(\rho_\kappa )\supseteq P_\kappa \W^{\o }$: Let $\Phi\in P_\kappa \W^{\o }$
and write $\Phi=(\Phi_1,\dots,\Phi_h)$. Fix $\lambda$ and represent
$\Phi_\lambda$ by a cocycle $z$ as above. It follows from
\cite[Lemma 6.3]{as} that $\int_X\varphi^{(m)}dz(f)=0$ for all $f\in
F_d^\lambda$ and all characteristic functions $\varphi^{(m)}$. Since
the function
$$((x,y),z)\mapsto \epsilon(x)z^{v}(xY-yX)^{k}$$ appearing
in the specialization map $\rho_{\kappa }$ can be written as an
uniform limit of functions $\varphi^{(m)}$, the inclusion follows.

\item $\Ker(\rho_\kappa )\subseteq P_\kappa \W^{\o }$: Let $c\in \Hom_{\bar
\Gamma^\lambda}(F^\lambda_k,\D_X)$ and choose $b$ such that
$c=T(p^m)b$: this is possible because $T(p)$ induces an isomorphism
on $\W^{\o }$ and, since $p$ is a principal ideal of $\r$, the
$T(p)$ operator preserves each of the cohomology groups $H^d
(\Gbar^\lambda,\D_X)$. Set $\pi:=\mat 100p$. Write
$\Gamma^\lambda\pi^m\Gamma^\lambda
=\coprod_t\Gamma^\lambda\alpha_{\lambda,t}$ and
$\gamma_{\lambda,t}:=\pi^{-m}\alpha_{\lambda,t}$. By \cite[Lemma
6.1]{as},
$$\int_X \varphi^{(m)}(y)\ dc(f)(y)= \sum_t
\int_X\varphi^{(m)} (y*\alpha_{\lambda,t}
)db(f\gamma_{\lambda,t}^{-1})(y).$$ Since $X_m\cap
X*\alpha_{\lambda,t}=\emptyset$ for $\alpha_{\lambda,t}\neq 1$ by
\cite[Lemma 6.6]{as}, the above sum is equal to
$$\int_{X_m}db(f\gamma_{\lambda,1}^{-1})(y)=\rho^*_\kappa
(\Phi_\lambda)(f\gamma_{\lambda,1}^{-1})=0.$$
From \cite[Lemma 6.3]{as}  if follows that $b$ takes values in
$P_\kappa \D_X$. Hence, $b$ belongs to the image of
$H^d(\Gbar^\lambda,P_\kappa \D_X)$ in $H^d(\Gbar^\lambda,D_X)$
which, by \cite[Lemma 1.2]{as}, is equal to $P_\kappa
H^d(\Gbar^\lambda,\D_X)$.
\end{enumerate}

Finally, from the equality $P_\kappa \W^{\o }=\Ker(\rho_\kappa )$
and a compact version of Nakayama's lemma, it follows that $\W^{\o
}$ is a $\Lambda$-module of finite type.\end{proof}

\subsection{Lifting system of eigenvalues}\label{section-prop-2}

Define:
\[\cl V_\alpha:=H^d_{\rm par}(X_{S(p\alpha)},K/\cl O),
\quad \cl V_\infty:=\dirlim_\alpha\cl V_\alpha,\]\[ \cl
V_\alpha^*:=\Hom_\cl O(\cl V_\alpha,K/\cl O) ,\quad \cl
V^*_\infty:=\Hom_\cl O(\cl V_\infty,K/\cl O)
\] where the direct
limit is computed with respect to the projection maps
$X_{S(p^\beta)}\to X_{S(p^\alpha)}$ for $\beta\geq\alpha$. The Hecke
algebras $h_{k,w}(S_0(p^\infty),\epsilon,\cl O)$ defined in \S
\ref{NOHA} can be equivalently introduced as $\invlim_\alpha
h^\prime_{k,w}(S_0(p^\alpha),\epsilon,\cl O)$, where
$h^\prime_{k,w}(S_0(p^\alpha),\epsilon,\cl O)$ is the image in
$\End_\cl O(H^d_{\rm par}(S(p^\alpha),\cl L(n,v,\epsilon,K/\cl O))$
of the algebra generated over $\cl O$ by the Hecke operators; the
same observation holds for the ordinary parts: see \cite[\S 3]{H6}
for details.

Let $P_\kappa $ be an arithmetic point corresponding to a
{primitive} form $f_\kappa $ of tame level $S$. Let $\cl R_{P_\kappa
}$ denote the localization of $\cl R$ at $P_\kappa $. Set $\cl
K_{P_\kappa }:={\rm{Frac}}(\cl R_{P_\kappa })$. Let $\cl
V^{*,\o}_\infty$ denote the ordinary submodule of $\cl
V^{*}_\infty$. For any arithmetic character $P_{n,v,\epsilon}$ which
factors through $\cl R_{P_\kappa }$, let and $\cl
V^{*,\o}_{\infty,P_{n,v,\epsilon}}$ denote the localization of $\cl
V^{*,\o}_\infty$ at $P_{n,v,\epsilon}$.

\begin{proposition} Let $d$ be odd.
Then $\cl V^{*,\o}_{\infty,{P_\kappa} }$ is free of rank $2^d$ over
$\cl R_{{P_\kappa} }$ and for each sign $\sgn\in\{\pm 1\}^d$, its
$\sgn$-eigenmodule is free of rank one.\end{proposition}

\begin{proof} It is enough to prove the second statement.
First note that as in \cite[page 1032]{H6} there is an isomorphism:
\[(\cl V^{*,\o}_{\infty,P_{n,v,\epsilon}})^\sgn/P_{n,v,\epsilon}(\cl
V^{*,\o}_{\infty,P_{n,v,\epsilon}})^\sgn \simeq
h_{k,w}^\o(S_0(p^\alpha),\epsilon,K)\] for any arithmetic point
$P_{n,v,\epsilon}$. By \eqref{control-theorem-hecke-algebras},
$k_{k,w}(S_0(p^\infty),\epsilon,K)$ is free of rank one over $\cl
R/P_{n,v,\epsilon}$, and hence, if $\fr m$ is the maximal ideal of
$\cl R_{{P_\kappa} }$, if follows that $(\cl
V^{*,\o}_{\infty,P_{n,v,\epsilon}})^\sgn/\fr m(\cl
V^{*,\o}_{\infty,P_{n,v,\epsilon}})^\sgn$ is free of rank one over
$\cl R_{{P_\kappa} }/\fr m$. The result follows from \cite[Lemma
3.10]{H2} and its proof.\end{proof}

Recall the following:
\begin{lemma}\label{meas-prof-spaces} Let $Z$ be a topological space
that is an inverse limit of finite discrete topological spaces
$Z_\alpha$ for $\alpha$ in some indexing set. Then the space of
$\scro$-valued measures on $Z$ is isomorphic to $\invlim
\rom{Fns}(Z_\alpha, \scro)$ where $\rom{Fns}(Z_\alpha, \scro)=
\scro^{Z_\alpha}$ is the space of continuous $\scro$-valued
functions on $Z_\alpha$.\end{lemma}
\begin{proof} If $\phi \in \invlim \rom{Fns}(Z_\alpha,\scro)$, then $\mu$
can be written as a compatible sequence of the form $\{\sum_{x \in
Z_\alpha} a_x \cdot x \}_\alpha$.  Let $p_\alpha: Z \to Z_\alpha$ be
the natural projection map and for each $x \in Z_\alpha$ set
$U_{\alpha, x} = p_\alpha^{-1} (x)$.  Then this is isomorphic to the
space of $\scro$-valued measures on $Z$ by the map
\begin{equation}\label{inv-lim-iso} \phi \mapsto \mu \textrm{ such that }
\mu(U_{\alpha, x}) = a_g.\end{equation} This defines a measure due
to the compatibility of the sequence.
\end{proof}

For each $\alpha$, let $p_\alpha: X_{S(p^\alpha)} \rightarrow X_S$.
Let $\scrf_\alpha = p_{\alpha *} \scro$ be the direct image of the
constant sheaf $\scro$ on $X_S$. Fix a point $x\in X_S$ and define
$Y_\alpha=Y_{\alpha,x}$ to be the fiber $p_\alpha^{-1}(x)$ of $x$
under $p_\alpha$. By Lemma \ref{meas-prof-spaces}, $\invlim
\scro^{Y_\alpha}$ is the space of $\scro$-valued measures on the
space $\invlim Y_\alpha$. Now, for the double coset $\GL_2(F)x
SS_\infty$ in $X_S$, there is a natural map
$$S/S(p^\alpha) = \GL_2(\r_p)/S(p^\alpha)_p \to Y_\alpha$$
given by $$zS(p^\alpha)_p \mapsto \GL_2(F) xzS(p^\alpha)S_\infty.$$
This map induces a bijection from $\GL_2(\r_p)/S(p^\alpha)(\r^\times
\cap SS_\infty)$ to $Y_\alpha$. Hence the inverse limit $\invlim
Y_\alpha$ can be identified with $\GL_2(\r_p)/NC$ which is finally
identified with $X$.

\begin{lemma}\label{lemma-sheaves}
The sheaves $\scrd_X$ and $\invlim_\alpha \scrf_\alpha$ on $X_S$ are
isomorphic.
\end{lemma}
\begin{proof} Let
$U \in X_S$ be an open set.  For each $\alpha$ and $w \in Y_\alpha$,
let $X_{\alpha, w} \subset X$ be the inverse image of $w$ under the
natural projection map $X \to Y_\alpha$. Let $u \in U$, then choose
$x \in \GL_2(F_\mathbb A)$ such that $u = \GL_2(\F) x S S_\infty$.
If $s \in \scrd _X(U)$ is a section in $\scrd_X$, we can express
$s(u) = \GL_2(\F) (x, \mu(x)) SS_\infty$ for some $\mu \in \D_X$,
depending on $s$. In fact, since $s$ is a locally constant section,
this expression is valid in a neighborhood $U_u$ of $u$.  Then
define a map:
$$ \begin{array}{ccl}
     \scrd_X (U) & \xrightarrow{\phi_\alpha (U)} &
     \scrf_\alpha (U)= \scro(p_\alpha^{-1}(U)) \\
      s & \mapsto & u \mapsto \sum_{w\in Y_\alpha} \mu(x)(X_{\alpha,w}).
   \end{array}$$  This map is independent of the choice of $x$ since
a different representative $x^\prime$ of the double coset $u$ would
yield the same measure $\mu(x)$ since $s$ is a section on $X_S$.

The compatible maps $\phi_\alpha$ then give rise to a map
$$\phi(U): \scrd_X(U) \to (\invlim_\alpha \scrf_\alpha) (U).$$ At
the level of the stalks this map is the isomorphism in
(\ref{inv-lim-iso}). It follows that $\phi$ is an isomorphism of
sheaves.\end{proof}

By Lemma \ref{lemma-sheaves},
\[\W=H^d_{\rm cpt}(X_S,\cl D_X)\simeq H^d_{\rm cpt}(X_S,\invlim
\cl F_\alpha).\] Since
\[H^d_{\rm cpt}(X_{S},\cl F_\alpha)=H^d_{\rm
cpt}(X_{S},p_{\alpha*}\cl O)\simeq H^d_{\rm cpt}(X_{S(p^\alpha)},\cl
O) ,\] there is a surjective  map
\begin{equation}\label{W-equation}\W\to\invlim_\alpha H^d_{\rm
cpt}(X_{S(p^\alpha)},\cl O).\end{equation}

By Poincaré duality:
\begin{equation}\label{poincare}H^d_{\rm
cpt}(X_{S(p^\alpha)},\cl O)\simeq H_d(X_{S(p^\alpha)},\cl
O).\end{equation}

By Pontryagin duality there is a canonical isomorphism:
\begin{equation}\label{pontryagin}
\Hom_{\cl O}(H^d(X_{S(p^\alpha)},K/\cl O),K/\cl O)\simeq
H_d(X_{S(p^\alpha)},\cl O),\end{equation} where $K:={\rm Frac}(\cl
O)$.

The injection \[\cl V_\alpha\hookrightarrow
H^d(X_{S(p^\alpha)},K/\cl O)\] induces an injection passing to the
direct limits:
\[\cl V_\infty\hookrightarrow \dirlim_\alpha H^d(X_{S(p^\alpha)},K/\cl
O)\] and hence there is a surjective map:
\begin{equation}\label{inv-lim}\Hom_\cl O(\dirlim_\alpha
H^d(X_{S(p^\alpha)},K/\cl O),K/\cl O)\to \cl
V_\infty^*.\end{equation}

By composing the maps \eqref{W-equation}, \eqref{poincare},
\eqref{pontryagin}, \eqref{inv-lim}, we get a surjective map:
\[\W\to \cl V^*_\infty.\] Note that this map is also
equivariant for the action of the Hecke operators and of the
involution $\tau=\left( \mat {1}00{-1},\dots,\mat {1}00{-1}
\right)$.

\begin{coro}\label{prop2}
The $h_\kappa $-eigenspace of $\W_\cl L\otimes_\cl L\cl K_{P_\kappa
}$ has dimension at least $2^d$ over $\cl K_{{P_\kappa}
}$.\end{coro}

\begin{proof} For each sign $\sgn$ there is an Hecke and $\tau$
equivariant map of finite dimensional $\cl K_{P_\kappa }$-vector
spaces
\[\W_\cl L\otimes_\cl L\cl K_{{P_\kappa} }\to \cl
V^{*,\o}_{\infty,{P_\kappa} }\otimes_{\cl R_{{P_\kappa} }}\cl
K_{{P_\kappa} }.\] Since $\cl H[\tau]$ is commutative and the
$(h_\kappa ,\sgn)$-eigenspace of $\cl V^{*,\o}_{\infty,{P_\kappa}
}\otimes_{\cl R_{{P_\kappa} }}\cl K_{{P_\kappa} }$ in non-trivial,
being of dimension 1 over $\cl K_{{P_\kappa} }$, it follows that the
$(h_\kappa ,\sgn)$-eigenspace in $\W_\cl L\otimes_\cl L\cl
K_{{P_\kappa} }$ is non-zero too (for this linear algebra argument,
see \cite[Lemma 5.10]{PS}). Since this holds for all $\sgn\in\{\pm
1\}^d$, the result follows.\end{proof}

\subsection{Proof of the Control
Theorem}\label{proof-control-theorem}

Before proving Theorem \ref{main-theorem}, recall the following
Lemma. Let $I_h:=\Ker(h)$ be the kernel of the canonical map $h:\cl
H\to\cl R$.

\begin{lemma}\label{ext-version-of-ash-stevens-lemma}
Let $M$ be an $\cl H$-module and $P$ be an ideal in $\Lambda$.
Suppose that $P$ is generated by an $M$-regular sequence
$(x_1,\dots,x_r)$.  Then the image of the map
$$i_*:\Ext_\scrh ^* (\scrh/I_h, PM) \to \Ext^*_\scrh (\scrh/I_h,M)$$
induced by the inclusion $i: PM \to M$ is equal to
$P\Ext_\scrh^*(\scrh/I_h, M)$.
\end{lemma}

\begin{proof} The proof of this lemma is a slight modification of that of
\cite[Lemma 1.2]{as} and will be omitted.\end{proof}

\begin{rem} In order to apply this lemma in the proof of the main
theorem, we first note that for any $\scrh$-module $M$, the Hecke
eigenspace $M^{h_0} = \mbox{Hom}_\scrh (\scrr, M) =
\mbox{Ext}_\scrh^0 (\scrr, M)$ and that $\W^\o$ is a $\scrh$ direct
summand of $\W$. Finally, we need to show that the generators of
$P_\kappa$ is a $\W^\o$-regular sequence. This follows from the fact
that the generators of $P_\kappa$ are a regular sequence for
$\Lambda_{P_kappa}$ and  that $\W^\o$ is a free module of finite
rank over $\Lambda_{P_\kappa}$.\end{rem}

Now the proof of the Control Theorem follows \cite{GS} and can be
described as follows. Recall that we have to show that for each
primitive arithmetic point $\kappa$ of weight
$(n(\kappa),v(\kappa))$ and character $\epsilon(\kappa)$, and for
choice of sign $\sgn \in \{ \pm \}^d$ the map
$$\rho_\kappa: (\W \otimes_{\Lambda}
\scrr_{P_\kappa})^{h_\kappa,\sgn}/P_\kappa (\W \otimes_{\Lambda}
\scrr_{P_\kappa})^{h_\kappa,\sgn} \rightarrow \W_\kappa$$ is an
isomorphism.

Since $h_\kappa (T(p))$ is a unit in $\scrr$, the module $(\W
\otimes_\Lambda {\cl R_{P_\kappa }})^{h_\kappa }$ is contained in
the nearly ordinary part $\W^\o \otimes_{\Lambda} {\cl R_{P_\kappa
}}$. Since, by Proposition \ref{first-prop}, $\W^\o$ is a free
$\Lambda$-module of finite rank, it follows that
$\W^\o\otimes_\Lambda\cl R_{P_\kappa }$ is a free $\cl R_{P_\kappa
}$-module of finite rank. By Proposition \ref{first-prop} and the
fact that $\cl R_{P_\kappa} $ is unramified over $\Lambda$, it
follows that the kernel of the map $\W^\o \otimes_\Lambda
\scrr_{P_\kappa } \rightarrow \W_{0}^\o$ is $P_\kappa (\W
\otimes_\Lambda \scrr_{P_\kappa })^\o $. This is the same as
$P_\kappa  (\W \otimes_{\Lambda } \scrr_{P_\kappa })^{h_\kappa }$ by
Lemma \ref{ext-version-of-ash-stevens-lemma} and the remark
following it. Hence, we get an injective map
\[(\W\otimes_\Lambda \scrr_{P_\kappa })^{h_\kappa }/P_\kappa  (\W
\otimes_{\Lambda } \scrr_{P_\kappa })^{h_\kappa }\to
\W_{\kappa}^{f_\kappa }.\]

Since $\W_\kappa ^{f_\kappa }$ is $2^d$-dimensional, to prove the
surjectivity of the map it suffices to show that $(\W
\otimes_\Lambda\scrr_{P_\kappa })^{h_\kappa }$ has $\cl
R_{P_\kappa}$-rank at least $2^d$. Recall that by Corollary
\ref{prop2} the $h_\kappa $-eigenspace of $\W_\cl L\otimes_\cl L\cl
K_{P_\kappa }$ has dimension $2^d$ over $\cl K_{P_\kappa }$. The
intersection of this eigenspace with $\W\otimes_{\Lambda}\cl
R_{P_\kappa }$ is a $\cl R_{P_\kappa }$-submodule of
$(\W\otimes_\Lambda{\cl R_{P_\kappa }})^{h_\kappa }$ of rank $2^d$.
The surjectivity of the above map follows.

Since specialization commutes with the action of the $2^d$-complex
conjugations, the result follows.

\section{$p$-adic $L$-functions}\label{section3}

\subsection{Complex $L$-functions}\label{sec-complex}

Let $f\in S_{k,w}(S(p^\alpha),\C)$ for some $\alpha\geq 0$. Recall
the modified Fourier coefficients $C(\fr m,f)$ defined in
\eqref{fourier}. Let $I_\fr c$ be the the group of fractional ideals
of $\r$ prime to $\fr c$ and fix a character $\chi:I_\fr
c\to\C^\times$ of sign $\sgn(\chi)$. For integral ideals $\fr m$
which are not prime to $\fr c$, let $\chi(\fr m):=0$. Define the
Dirichlet series:
$$L(f,\chi,s):=\sum_{\m }\frac{\chi(\fr m)C(\m ,f)}{\norm{\fr m}^s},$$
where the sum is over all integral ideals of $\r$ and $\norm{\fr m}$
is the norm of $\fr m$. By \cite{Sh}, this $L$-series converges if
$\Re(s)$ is sufficiently large and can be continued analytically to
an entire function on $\C$. Choose a finite index subgroup
$U\subseteq\r_+^\times$ such that
$a_\lambda(\epsilon\xi)=\epsilon^{k/2}a_\lambda(\epsilon\xi)$ for
all $\epsilon\in U$, $\xi\in\fr t_\lambda\fr d$ and $\lambda$. The
choice of $U$ is possible by \cite[$1.8_b$]{Sh}. Define
\[
L(f_\lambda,\chi,s)=[\r_+^\times:U]^{-1}\sum_{\xi\in\fr t_\lambda\fr
d/U,\xi\gg 0}\frac{\chi(\xi(\fr t_\lambda\fr
d)^{-1})a_\lambda(\xi)\xi^v}{b_{v,\lambda}\norm{\xi(\fr t_\lambda\fr
d)^{-1}}^s}.\] (Recall that $b_{v,\lambda}$ is defined in
\eqref{fourier}.) Then
$$L(f,\chi,s)=\sum_{\lambda=1}^{h(p^\alpha)}L(f_\lambda,\chi,s).$$

\subsection{Special values}

Let $\r_+^\times$ the group of totally positive units in
$\r^\times$. For any element $x\in F$, choose a subgroup
$\r_x\subseteq \r_+^\times$ of finite index such that $ex\equiv x
\pmod{\fr t_\lambda^{-1}}$ for all $e\in\r_x$. The group
$\r_+^\times$ act on $\R_+[I]$ by \[x=\sum_{\sigma\in
I}x_\sigma\sigma\mapsto ex:=\sum_{\sigma\in
I}\sigma(e)x_\sigma\sigma.\] Let
$D:=\left\{x\in\R_+[I]:\prod_{\sigma\in I}x_\sigma=1\right\}.$ Since
$\prod_{\sigma\in I}\sigma(e)=1$ for all $e\in \r^\times_+$, the
action of $\r_{x}^\times$ on $\R_+[I]$ restricts to an action on
$D$. The logarithm map $\log:\R_+[I]\to\R[I]$ defined by
$\log(\sum_{\sigma\in I}x_\sigma\sigma):=\sum_{\sigma\in
I}\log(x)_\sigma\sigma$ is an isomorphism and the action of
$\r_x^\times$ on the left corresponds to the translation by the
sublattice $\log(\r_x^\times)$ of rank $d-1$ on the right. Choose a
fundamental parallelogram of this lattice and denote by $B(x)$ its
inverse image under $\log$. Define the cone $C(x)$ with vertex
$0\in\R[I]$ and base $B(x)$ to be
\[C(x):=\left\{t\xi=\sum_{\sigma\in I}t\xi_\sigma\sigma:
0<t<\infty,\xi=\sum_{\sigma\in I}\xi_\sigma\sigma\in B(x)\right\}.\]
Then $C(x)$ is a fundamental domain in $\R_+[I]$ for $\R_+[I]/\r_x$.

Let $E$ be a $\Gamma^\lambda(p^\alpha)$-module. Suppose that $E$ is
also a $\Z[T_x,x\in F]$-module, where $T_x:=\mat 1x01$. For any
modular symbol $\omega\in H^d_\cpt(\Gamma^\lambda(p^\alpha)\bk \fr
H,\cl E)$, where $\cl E$ is the coefficient system associated to
$E$, and any $x\in F$, define
$$L(\omega,C(x)):=\frac{\int_{x+iC(x)}\omega|_ET_x}{[\r_+^\times:\r_x]}.$$
In the above formula $|_E$ is the action of  of $T_x$ on $E$. For
$x=0$, write simply $L(\omega)$ for $L(\omega,C(0))$.

For any finite order character $I_\fr c\to\C^\times$, define:
$$L(\omega,\chi):=\sum_{a\in \fr c^{-1}\fr
t_\lambda^{-1}/\fr t_\lambda^{-1}}\sgn(a)^{\sgn(\chi)}\chi(a\fr c\fr
t_\lambda)L(\omega,C(a)).$$

\begin{rem} Note that in general
$L(\omega,C(x))$ \emph{depends} on the choice of the fundamental
domain $C(x)$. However, the values $L(\omega,C(x))$ appearing in the
arithmetic applications do not depend on the choice of $C(x)$, as
will be clarified in the following.\end{rem}

\subsection{Interpolation formulas
for classical modular symbols}\label{IFCMS}

Let $f=(f_\lambda)_{\lambda=1}^{h(p^\alpha)}\in
S_{k,w}(S(p^\alpha),\C)$ for some $\alpha\geq 0$. Use the following
notations: for any $r\in\Z[I]$ with $0\leq r_\sigma\leq n_\sigma$,
set $\left(\begin{array}{c}n\\r\end{array}\right):=\prod_{\sigma\in
I}\left(\begin{array}{c}n_\sigma\\r_\sigma\end{array}\right)$ and
$|r|:=\sum_\sigma n_\sigma$. Write
$$L(\omega(f_\lambda),C(x))=\sum_{r=0}^n\left(\begin{array}{c}n\\r\end{array}\right)(-1)^{r+t}
L(\omega(f_\lambda),C(x),r)X^rY^{n-r}.$$ For $x=0$, write simply
$L(\omega(f_\lambda),r)$ for $L(\omega(f_\lambda),C(0),r)$.
Likewise, write
$$L(\omega(f_\lambda),\chi)=\sum_{r=0}^n\left(\begin{array}{c}n\\r\end{array}\right)(-1)^{r+t}
L(\omega(f_\lambda),\chi,r)X^rY^{n-r}.$$ In the above formulas, the
sum is over all $r\in\Z[I]$ such that $0\leq r_\sigma\leq n_\sigma$.

Set $k_0:=\max\{k_\sigma,\sigma\in I\}$ and
$k_{-1}:=\min\{k_\sigma,\sigma\in I\}$. Define the \emph{critical
strip} for $f$ to be the set of complex numbers $s\in\C$ such that
\begin{equation}\label{critical-strip}
k_*:=\frac{k_0-k_{-1}}{2}+1\leq\Re(s)\leq
k^*:=\frac{k_0+k_{-1}}{2}-1.\end{equation}

It is easy to show that for any integer $m$ in the critical strip,
the values $L_{m-v-1}(\omega(f_\lambda),C(x))$ does not depend on
the choice of the fundamental domain for $\R_+[I]/\r^\times_x$ (see
\cite[\S 3.4]{Ma}).

Set:
$$L(\omega(f),r):=\sum_{\lambda=1}^{h(p^\alpha)}
L(\omega(f_\lambda),r)\norm{\fr t_\lambda\fr
d}^{-r+v+1}b_{v,\lambda}.$$ For any $n=\sum n_\sigma\sigma\in\Z[I]$,
define $\Gamma_F(n):=\prod_{\sigma\in I}\Gamma(n_\sigma)$, where
$\Gamma$ is the usual complex $\Gamma$-function.

\begin{proposition} For any integer $k_*\leq m\leq k^*$ in the critical
strip, the following interpolation formulas hold:
$$L(\omega(f_\lambda),m-v-1)=c(m,v,\lambda){L(f_\lambda,m)\Gamma_F(m-v)}$$
where $c(m,v,\lambda,):=i^{|m-v-t|}(2\pi)^{|v-mt|}\norm{\fr
t_\lambda\fr d}^{m}b_{v,\lambda}$, and
$$L(\omega(f),m-v-1)=\frac{{i}^{|m-v-1|}L(f,m)\Gamma_F(m-v)}{(2\pi)^{|mt-v|}}.
$$
\end{proposition}

\begin{proof} It is enough to show the first formula, which can be obtained by following
\cite[\S4]{Sh}. Choose $U$ a finite index subgroup of $\r_+^\times$
as in \S \ref{sec-complex} such that
$a_\lambda(\epsilon\xi)=\epsilon^{k/2}a_\lambda(\epsilon\xi)$ for
all $\epsilon\in U$ and all $\lambda$. First note that
\[\int_{C(0)}f_\lambda(iy)y^{m-v-t}dy=
{\int_{\R^d_+/U }\sum_{e\in U}\sum_{\xi\in\fr t_\lambda\fr
d/U,\xi\gg 0} a_\lambda(e\xi)e^{-2\pi(e\xi\cdot y)}y^{m-v-t}dy}\]
(here quotient measures are implicitly used). Set
$a'_\lambda(\xi):=a_\lambda(\xi)\xi^v$. Since
$a'_\lambda(e\xi)=a'(\xi)$ for all $e\in U$, it follows that:
\[
\int_{C(0)}f_\lambda(iy)y^{mt-v-t}dy =\sum_{\xi\in\fr t_\lambda\fr
d/U,\xi\gg 0}\frac{a'_\lambda(\xi)}{\xi^{mt-1}}\int_{\R^d_+}
e^{-2\pi(\xi\cdot y)}(\xi y)^{mt-v-t}dy.
\]
Changing variables $s:=2\pi(\xi_\sigma y_\sigma)_\sigma$ yields
\[
\int_{\R^d_+} e^{-2\pi(\xi\cdot y)}(\xi y)^{mt-v-t}dy= \int_{\R^d_+}
e^{-s}\frac{s^{mt-v-t}}{(2\pi)^{|mt-v-t|}}\frac{ds}{(2\pi)^{|t|}\xi},
\]
hence:
\[
\int_{C(0)}f_\lambda(iy)y^{m-v-t}dy={(2\pi)^{|v-mt|}}\sum_{\xi\in\fr
t_\lambda\fr d/U,\xi\gg
0}\frac{a'_\lambda(\xi)}{\xi^{mt}}\Gamma_F(m-v).
\]
The result follows.\end{proof}

Fix a character $\chi$ with sign $\sgn(\chi)$ as in \S
\ref{sec-complex}. Assume that $\chi$ is primitive modulo $\fr c$.
Form the Gauss sum:
$$\tau(\chi):=\sum_{u\in\fr c^{-1}/\cl O_F}\sgn(u)^{\sgn(\chi)}\chi(u\fr c)e^{2\pi i
(u\cdot t)}.$$ If $\chi$ is not primitive modulo $\fr c$, let
$\chi'$ be the primitive associated character and define
$\tau(\chi):=\tau(\chi')$. Define:
\[f_\lambda^\chi(z):=\sum_{x\in\fr t_\lambda\fr d, \xi\gg 0}\chi(x(\fr
t_\lambda\fr d)^{-1})a_\lambda(x)e^{2\pi i(x\cdot z)}.\] By
\cite[Proposition 4.4]{Sh}, $f_\lambda^\chi$ is a modular form for
the same congruence subgroup and of the same weight as $f_\lambda$,
with character $\epsilon\chi^{-2}$, and
$$f_\lambda^\chi(z)={\tau(\overline\chi)^{-1}}\sum_{u\in\fr c^{-1}\fr t_\lambda^{-1}/\fr
t_\lambda^{-1}}
\rom{sgn}(u)^{\sgn(\chi)}\overline \chi(u\fr c\fr
t_\lambda)f_\lambda(z+u).$$ Combining this expression with the
definition of twisted $L$-function and the interpolation formula
described above yields:
\[\sum_{u\in\fr c^{-1}\fr t_\lambda^{-1}/\fr
t_\lambda^{-1}}\rom{sgn}(u)^{\sgn(\chi)} \overline\chi(u\fr c\fr
t_\lambda)L(\omega(f_\lambda),C(u),m-v-1)=\]
\begin{equation}\label{formula-chi}=\frac{c(m,v,\lambda)}{\tau(\chi)}
\Gamma_F(m-v)L(f_\lambda,\chi,m).\end{equation}

\begin{rem} In the above formula we
understand that a finite index subgroup $\fr l\subseteq \r^\times_+$
has been chosen such that $\fr l\subseteq\r_u$ for all $u$. Use the
independence of $L(\omega(f_\lambda),C(u))$ on the choice of $\fr l$
as above to recover formula \eqref{formula-chi}. We will not mention
this point hereafter.\end{rem}

Setting
\[L(\omega(f),\chi,r):=\sum_{\lambda=1}^{h(p^\alpha)}L(\omega(f_\lambda),\chi,r)
\norm{\fr t_\lambda}^{-m+v+1}b_{v,\lambda}\]
\begin{equation}\label{Lambda}\Lambda(f_\lambda,\chi,m):=\frac{c(m,v,\lambda)}{\tau(\chi)}
\Gamma_F(m-v)L(f_\lambda,\overline\chi,m)\end{equation}
\begin{equation}\label{Lambda-1}\Lambda(f,\chi,m):=\sum_{\lambda=1}^{h(p^\alpha)}\Lambda(f_\lambda,\chi,m)\norm{\fr
t_\lambda}^{-r+v+1}b_{v,\lambda}\end{equation} yields formulas in a
more compact form:
\begin{equation}
L(\omega(f_\lambda),\chi,m-v-1)=\Lambda(f_\lambda,\chi,m) \mbox{ and
} L(\omega(f),\chi,m-v-1)=\Lambda(f,\chi,m).\end{equation}

Fix a sign $\sgn=\{\sgn_\sigma,\sigma\in I\}$, a character $\chi$ of
sign $\sgn(\chi)=\sgn$ and an eigenform $f\in
S_{k,w}(S(p^\alpha),K_f)$ with eigenvalues in $K_f$. Denote by
$K_f(\chi)$ the extension of $K_f$ generated by the values of
$\chi$. By \cite{Sh}, for any sign $\sgn'\in\{\pm 1\}^I$ there
exists periods $\Omega^{\sgn'}(f)$ such that for any integer $m$ in
the critical strip for the weight of $\kappa$:
\begin{equation}\label{Omega}\frac{L(\omega(f),\chi,m-v-1)}{\Omega^{\sgn(m)}(f,\chi)}\in
K(\chi),\end{equation} where $\sgn(m):=(-1)^m\sgn$ (see \cite[\S
8.1]{Pa} for a result stated in a form close to this). Define the
normalized modular symbols associated to $f$ to be
\[\Psi^{\sgn}_{f_\lambda}:=\omega(f)/\Omega^{\sgn}(f,\chi),\] so that
\begin{equation}\label{formula-chi-bis}
L(\Psi_{f_\lambda}^{\sgn},\chi,m-v-1)=\frac{\Lambda(f_\lambda,\chi,m)}{\Omega^\sgn(f,\chi)},\end{equation}
where
$L(\Psi_{f_\lambda}^{\sgn},\chi,r):=\frac{L(\omega({f_\lambda})^{\sgn},\chi,r)}
{\Omega^{\sgn}(f,\chi)}$ for $r=m-v-1$ and $m$ in the critical
strip. Setting $L(\Psi_f^{\sgn},\chi,r):=\sum_\lambda
L(\Psi_{f_\lambda}^{\sgn},\chi,r)\norm{\fr
t_\lambda}^{-r+v+1}b_{v,\lambda}$ as usual for $r$ as above yields:
\begin{equation}
L(\Psi_f^{\sgn},\chi,m-v-1)=\frac{\Lambda(f,\chi,m)}{\Omega^\sgn(f,\chi)}.\end{equation}

\subsection{$p$-adic $L$-functions of $\Lambda$-adic
modular symbols}\label{section4.3}

The map $z\mapsto iz$ defines a map: $\Delta_x^\lambda:C(x)\to X_S$
which induces a corresponding map on the cohomology $H^d_\rom{cpt}
(X_S,\cl D_X) \rightarrow H^d_\rom{cpt}(C(x), \Delta^\lambda_x\cl
D_X)$, denoted by the same symbol. For any $\Phi\in\W$, write
$\Phi=(\Phi_1,\dots,\Phi_h)$ and define the \emph{special value of
$\Phi$} to be $L(\Phi)=(L(\Phi_1),\dots,L(\Phi_h))$ where for each
$\lambda=1,\dots,h$:
$$L(\Phi_\lambda,C(0))=L(\Phi_\lambda):=\int_{C(0)}\Delta^\lambda_1 \Phi_\lambda\in D.$$

Define: $$X^{\prime\prime}:=NC\bk\left\{\left(\mat abcd\in
X^\prime:d\in\r_p^\times\right)\right\}.$$ Then $X^{\prime\prime}$
is identified with $\ec\bk(\r_p^\times
\times\r_p^\times\times\r_p^\times)$ under the bijection between $X$
and $\ec\bk((\r_p^2)^\times\times\r_p^\times)$.

For any $p$-adic group $\Pi$, set $\cl
X(\Pi):=\Hom_{\rom{cont}}(\Pi,\overline\Q_p^\times)$. Any $P\in\cl
X(G)$ induces a character on
$\ec\bk(\r_p^\times\times\r_p^\times)$ via the isomorphism
$G\simeq(\r_p^\times\times\r_p^\times)/\ec$. In particular, if
$\kappa\in\cl A(\cl I)$ has weight $(n,v)$ and character
$\epsilon$, \[\kappa((x,z))=\epsilon\chi_\cyc^n(x)z^v.\] Fix
$\Phi\in\W$. For $(\kappa,\sigma)\in\cl X(G)\times\cl
X(\r_p^\times)$, define the \emph{standard several variables
$p$-adic $L$-function of} $\Phi_\lambda$:
$$L_p(\Phi_\lambda, \kappa, \sigma) = \int_{X^{\prime\prime}}
\kappa(x,z/x^2)\sigma(y/x)\ dL(\Phi_\lambda)(x,y,z).$$

Fix two multi-integers $r=(r_\sigma)_\sigma\in\Z[I]$ and
$m=(m_\sigma)_\sigma\in\Z[I]$ with all $r_\sigma$ and $m_\sigma$ non
negative. Write $p^m$ for $\prod_{\sigma\in I}\fr
p_\sigma^{m_\sigma}$, where $\frakp_\sigma$ is the prime divisor of
$\iota\circ\sigma(p)$. Say that a point $\sigma\in\cl
X(\r_p^\times)$ is \emph{arithmetic} if there exist a character
$\chi:I_{p^m}\to\overline\Q_p^\times$ such that
$\sigma(x)=\sigma_{\chi,r}(x)=\chi(\overline x)x^r$, where
$x\mapsto\overline x$ is the projection map $\r^\times_p\to
\r^\times/p^m$ and $m,r$ are as above. Denote by $\cl
A(\r_p^\times)$ the set of arithmetic character. Define the
\emph{sign} of $\sigma_{\chi,m}$ to be the sign of $\chi$. If
$\sigma_{\chi,r}\in\cl A(\r_p^\times)$, then it can be extended to a
character $\sigma\in\Hom_{\rom{cont}}(\r_p,\overline\Q_p)$ by
setting $\sigma(x)=x^r$ if $\chi$ is the trivial character and
$\sigma(x):=0$ for $x\not\in \r_p^\times$ otherwise. For
$(\kappa,\sigma)\in\cl X(G)\times \cl A(\r_p^\times)$, define the
\emph{improved several variables $p$-adic $L$-function of}
$\Phi_\lambda$:
$$L_p^* (\Phi_\lambda, \kappa, \sigma) = \int_{X^\prime}
\kappa(x,z/x^2)\sigma(y/x)\ dL(\Phi_\lambda)(x,y,z).$$

It follows from the definitions that
$L_p(\Phi_\lambda,\kappa,\sigma)$ is analytic in
$(\kappa,\sigma)\in\cl X(G)\times\cl X(\r_p^\times)$ and
$L_p^*(\Phi_\lambda,\kappa,\sigma)$ is analytic in
$(\kappa,\sigma)\in\cl X(G)\times\cl A(\r_p^\times)$. Define
homomorphism of $\widetilde\Lambda$-algebras
\[L(\cdot,\sigma):\W\to\widetilde\Lambda \quad \mbox{ and }\quad
L^*(\cdot,\sigma):\W\to\widetilde\Lambda\] by requiring that
\[\kappa(L_p(\Phi_\lambda,\sigma))=L_p(\Phi_\lambda,\kappa,\sigma)
\quad \mbox{ and }\quad
\kappa(L_p^*(\Phi_\lambda,\sigma))=L_p^*(\Phi_\lambda,\kappa,\sigma).\]
For any continuous $\widetilde{\Lambda}$-algebra $R$, extending by
$R$-linearity the homomorphisms $L(\cdot,\sigma)$ and
$L^*(\cdot,\sigma)$ yields homomorphisms of $R$-algebras
\[L_p(\cdot,\sigma)_R:\W_R\to R\quad\mbox{ and }\quad L_p^*(\cdot,\sigma)_R:\W_R\to R.\]
For any
$\Phi_\lambda\in \W_R$, define standard and improved $p$-adic
$L$-functions:
$$L_p(\Phi_\lambda,\kappa,\sigma):=\kappa(L_p(\Phi_\lambda,\sigma)_R), \mbox{ for }
(\kappa,\sigma)\in \cl X(R)\times\cl X(\r_p^\times),$$
$$L_p^*(\Phi_\lambda,\kappa,\sigma):=\kappa(L_p^*(\Phi_\lambda,\sigma)_R), \mbox{ for }
(\kappa,\sigma)\in \cl X(R)\times\cl A(\r_p^\times).$$

Fix a prime divisor $\fr p$ of $p$ and an index
$\lambda\in\{1,\dots,h,\}$. Let $\mu$ such that $\fr p t_\lambda
t_\mu^{-1}$ is trivial in the strict class group of $F$. Let
$\alpha_\lambda$ such that $S\pi_\fr p
S=Sx_\lambda^{-1}\alpha_\lambda x_\mu S$ and form the coset
decomposition $\Gamma^\lambda\pi^m_\fr p\Gamma^\mu
=\coprod_t\Gamma^\lambda\alpha_{\lambda,t}$. The matrices
$\alpha_{\lambda,t}$ can be written in the form
$\alpha_\infty:=\left(\begin{array}{cc}\alpha^m&0\\0&1\end{array}\right)$
or
$\alpha_a:=\left(\begin{array}{cc}1&a\\0&\alpha^m\end{array}\right)$
with $\val_\fr p(\alpha)=1$ and $a$ varying over a set
$\Sigma(\frakp^m)$ of representatives of $\fr t_\lambda^{-1}/\fr
p^m\fr t_\lambda^{-1}$. Since $\fr t_\lambda$ is prime to $p$,
$\alpha_\infty$ and $\alpha_a$ for $a$ as above all belong to
$\M_2(\r_p)$. Let $U(t,\fr p^m):=X\cap Y^\prime*\alpha_{t}$ for
$t\in\Sigma(\fr p^m)\cup\{\infty\}$ and denote by ${\D}(t,\frakp^m)$
the subset of ${\D}$ consisting of those measures supported on
$U(t,\frakp^m)$. Then $\Phi_\lambda*T(\frakp)^m=\sum_t\Phi_{\mu,t}$
where $\Phi_{\mu,t}\in
H^d_{\rom{cpt}}(\Gamma^\mu,{\D}(t,\frakp^m))$. Let
$\Sigma^\times(\frakp^m)$ denote the subset of $\Sigma(\frakp^m)$
consisting of elements which are prime to $p$.

\begin{lemma}\label{lemma3} Let $(\Phi_\lambda)_\lambda\in\W_R$ for
a continuous $\wt\Lambda$-algebra $R$. Then:
$$L_p(\Phi_\lambda*T(\frakp^m),\kappa,\sigma)=
\sum_{a\in\Sigma^\times(\frakp^m)}\int_{X^\prime}\kappa(x,z/x^2)
\sigma(a+\alpha^my/x)L(\Phi_{\mu,a})(x,y)$$
$$L_p^*(\Phi_\lambda*T(\frakp^m),\kappa,\sigma)=
\sum_{a\in\Sigma(\frakp^m)}\int_{X^\prime}\kappa(x,z/x^2)\sigma(a+\alpha^my/x)L(\Phi_{\mu,a})(x,y).$$\end{lemma}
\begin{proof} Note that
$L(\Phi_\lambda*T(\frakp^m))=\sum_tL(\Phi_{\mu,t})$ where
$t\in\Sigma(\frakp^m)\cup\{\infty\}$. Since the standard and
improved $L$-functions are defined as integrals over $X^\prime$ and
$X^{\prime\prime}$, $\alpha_\infty$ does not contribute to the
integral. A simple computation shows that the action of
$\left(\begin{array}{cc}1&a\\0&\frakp^m\end{array}\right)$ on the
characteristic function of $U(a,\frakp^m)$ is the characteristic
function of $X^\prime$. The result follows.\end{proof}

\begin{proposition}\label{prop4.6}Let $(\Phi_\lambda)_\lambda\in\W_R$ for
a continuous $\wt\Lambda$-algebra $R$. Then:
$$L_p(\Phi_\lambda*T(\frakp ),\kappa,\sigma)=L_p^*(\Phi_\lambda*T(\frakp ),\kappa,\sigma)-
\sigma(\frakp )L^*_p(\Phi_\mu,\kappa,\sigma).$$
\end{proposition}

\begin{proof}
Put $m=1$ in the formulas of Lemma \ref{lemma3} and note that the
only term surviving in the difference is that for $a=0$.
\end{proof}

Fix $\theta:\cl R\to\cl I$ and $\kappa\in\cl A(\cl I)$ an arithmetic
point. Let $P:=\kappa_{|\Lambda}$.
Recall the notations
\[\Phi_{\lambda,\theta_k}=\rho_\kappa(\Phi_\lambda)\] introduced in
\eqref{Phi-theta-kappa}, where $\rho_\kappa:\W\otimes_\Lambda\cl R_{P}\to\W_\kappa$
is the specialization map. Let $\chi$ be a finite order character of
$I_{(p^m)}$ and write
\begin{equation}
L(\Phi_{\lambda,\theta_\kappa},\chi)
=\sum_{r=0}^{n}\left(\begin{array}{c}{n}\\r\end{array}\right)(-1)^{r}
L(\Phi_{\lambda,\theta_\kappa},\chi,r)X^{r}Y^{n-r}.\end{equation}

\begin{lemma}\label{l} Notations as above. Let
$\sigma=\sigma_{\chi,r}$ with $r\in\Z[I]$ and $\chi$ a finite order
character of $I_{\frakp^m}$. Suppose that $r_\frakp\leq n_\sigma$
for all $\sigma$ associated to $\fr p$. Fix $\lambda$ and let $\mu$
defined as before Lemma \ref{lemma3}. Then
$$L^*(\Phi_\lambda*T(\frakp^m),\kappa,\sigma)=
L(\Phi_{\mu,\theta_\kappa},\chi,r).$$\end{lemma}

\begin{proof}
To simplify notations set $\chi'(a):=\sgn(a)^{\sgn(\chi)}\chi(a\fr
t_\lambda)$. Compute from the equation of Lemma \ref{lemma3}:
$$
L^*_p(\Phi_\lambda*T(\frakp^m),\kappa,\sigma)=$$
$$=\sum_{a\in\Sigma(\frakp^m)}\chi'(a)\int_{X^\prime}z^v\epsilon(x)
x^{n-r}(ax+\alpha^my)^rdL(\Phi_{\mu,a})(x,y,z).$$ Hence
$$\sum_{r=0}^{n}\left(\begin{array}{c}n\\r\end{array}\right)(-1)^r
L^*_p(\Phi*T(\frakp^m),\kappa,\sigma)X^rY^{k-2-r}=$$
$$=\sum_{a\in\Sigma(\frakp^m)}\chi'(a)\int_{X^\prime}z^v\epsilon(x) (x(Y-aX)-y\alpha^mX)^n
dL(\Phi_{\mu,a})(x,y,z).$$ On the other hand, since
\[L(\Phi_{\mu,\theta_\kappa},\chi,r)=
\sum_{a\in\Sigma(\frakp^m)}\chi^\prime(a)L(\Phi_{\mu,\theta_\kappa},0)|\mat
1a0{\alpha^m},\] it follows that
$$\sum_{r=0}^{n}\left(\begin{array}{c}n\\r\end{array}\right)
(-1)^rL(\Phi_{\mu,\theta_\kappa},\chi,r)X^rY^{k-2-r}=$$
$$=\sum_{a\in\Sigma(\frakp^m)}\chi'(a)\int_{X^\prime}
z^v\epsilon(x)(xY-yX)^{n}dL(\Phi_{\mu,a})(x,y,z)|\mat
1a0{\alpha^m}.$$ Comparing the two displayed quantities yields the
result.\end{proof}

\begin{proposition}\label{prop4.7}
Notations as above. Let
$\sigma=\sigma_{\chi,r}$ with $r\in\Z[I]$ and $\chi$ a finite order
character of $I_{p^m}$ for some $m\in\Z[I]$. Suppose that
$r_\frakp\leq n_\sigma$ for all $\sigma$ associated to $\fr p$. Let
$\lambda$ and $\mu$ be as in Lemma \ref{l}. Then
$$L^*(\Phi_\lambda*T(p^m),\kappa,\sigma)=
L(\Phi_{\mu,\theta_\kappa},\chi,r).$$\end{proposition}

\proof Follows from Lemma \ref{l}.\endproof

Recall the notations of \S \ref{NOHA}: $\cl
R=h_{2t,t}^\ord(S(p^\infty),\cl O)$ and for any $\kappa\in\cl X(\cl
I)$ with $\kappa_{|\Lambda}=:P$, $\cl R_P$ is the localization of
$\cl R$ at the kernel of $P$. Set $\cl K:=\cl R\otimes_\Lambda\cl L$
and $\cl K_P:=\rom{Frac}(\cl R_P)$ (fraction field). Since $\cl K_P$
is a direct factor of $\cl K$, it follows that $\W_{\cl K_P}$ is a
direct factor of $\W_\cl K$. Say that
$\Phi\in\W_{\cl K}$ is \emph{regular} at $\kappa$
if the projection of $\Phi$ to $\W_{\cl K_P}$
belongs o the $\cl R_P$-submodule $\W_{\cl R_p}$. Every
$\Phi\in\W_\cl K$ is regular at all but a finite
number of $\kappa\in\cl X(\cl I)$.

Fix $\theta:\cl R\to\cl I$ and ${\kappa_0}\in\cl A(\cl I)$ as above. Let
$\Phi\in\W_{R_{P_0}}$ where recall that $P_0={\kappa_0}_{|\Lambda}$.  and define
the \emph{domain of convergence $U(\Phi,\theta_{\kappa_0})$ of
$\Phi$ near $\theta_{\kappa_0}$} to be the set of $\kappa\in\cl X(\cl
I)$ such that $\Phi_\lambda$ is regular at $\kappa$. For any
$\kappa\in U(\Phi,\theta_{\kappa_0}) $, $\chi$ a character of sign
$\sgn$ as above and $m$ an integer, define:
$$L_p^\sgn(\Phi_\lambda,\theta,\kappa,\chi,m):=L_p(\Phi_\lambda,\kappa,\chi\chi_\cyc^{m-v-1})$$
$$L_p^{*,\sgn}(\Phi_\lambda,\theta,\kappa,\chi,m):=L_p^*(\Phi_\lambda,\kappa,\chi\chi_\cyc^{m-v-1}).$$
Here by abuse of notations, $\chi_\cyc^r$ denotes the character
$x\mapsto x^r$ for any $r\in\Z[I]$. Although these two functions are
defined for $m\in\Z$, they can be extended in a unique way to
functions $L_p^\sgn(\Phi_\lambda,\theta,\kappa,\sigma)$ for
$\sigma\in\cl X(\r_p^\times)\cap U(\Phi,\theta_{\kappa_0}) $ and
$L_p^{*,\sgn}(\Phi_\lambda,\theta,\kappa,\sigma)$ for $\sigma\in\cl
A(\r_p^\times)\cap U(\Phi,\theta_{\kappa_0}) $ such that for
$\sigma=\chi\chi_\cyc^m$,
\[L_p^\sgn(\Phi_\lambda,\theta,\kappa,\sigma)=L_p^\sgn(\Phi
,\theta,\kappa,\chi,m)\]
\[L_p^{*,\sgn}(\Phi_\lambda,\theta,\kappa,\sigma)=
L_p^{*,\sgn}(\Phi_\lambda,\theta,\kappa,\chi,m).\]

Let $\theta$ and $\kappa_0$ be fixed as above. Fix a sign $\sgn\in\{\pm 1\}^d$.
Choose by the Control Theorem
\ref{main-theorem} a modular symbol $\Phi\in\W_{\cl R_{P_0}}^\sgn$ such that
\begin{equation}\label{Phi}
\rho_{\kappa_0}(\Phi)=\Psi^\sgn_{f_{\kappa_0}}.\end{equation}
The Control Theorem \ref{main-theorem}
implies the existence of a unique period $\Omega(\Phi,\theta,\kappa)$ such that
\begin{equation}\label{eq-30}\Phi_{\lambda,\theta_\kappa}=
\Omega(\Phi,\theta,\kappa)\Psi_{f_{\kappa,\lambda}}^{\sgn}.\end{equation}

\begin{proposition}\label{prop4.9} Let
$\theta$, $\kappa_0$ and $\Phi$ be fixed as above such that \eqref{Phi} holds.
For any $\lambda=1,\dots,h$, any arithmetic character $\kappa\in\cl
A(\cl I)\cap U(\Phi,\theta_{\kappa_0})$, any character $\chi$ of sign
$\sgn$ as above and any positive integer $m$ in the critical strip
for the weight of $f_{\kappa}$:
$$L_p^\sgn(\Phi_\lambda,\theta,\kappa,\chi,m)=e(p,\chi,m)
L_p^{*,\sgn}(\Phi_\lambda,\theta,\kappa,\chi,m),$$ where
\[e(p,\chi,m):=\prod_{\fr p\mid p}\left(1-\frac{\chi(\frakp)
\chi_\cyc^{m-1}(\frakp)}{a_\frakp(\theta,\kappa)}\right)
\quad \mbox{ and } \quad
a_\frakp(\theta,\kappa):=\theta_\kappa(T(\frakp)).\]
\end{proposition}

\begin{proof} This follows from Proposition \ref{prop4.6}
and behavior of the specialization map described in Proposition
\ref{hecke/actions}, part \ref{part-2} (note that
$\Phi_\lambda*T(\fr p)=a_\frakp(\theta,\kappa)\Phi_\mu$).\end{proof}

\begin{proposition}\label{prop4.8}
Notations and assumptions as in Proposition \ref{prop4.9}. Then
$$L(\Phi_{\lambda,\theta_\kappa},\chi,m)=
\frac{\Omega(\Phi,\theta,\kappa)}{\Omega^\sgn(f_{\kappa},\chi)}\Lambda(f_{\kappa,\lambda},\chi,m).$$
\end{proposition}
\begin{proof} This formula follows by combining equations
\eqref{formula-chi-bis} and \eqref{eq-30}.\end{proof}

\begin{theorem} \label{THM-1}
Notations and assumptions as in Proposition \ref{prop4.9}. Then
$$L_p^{*,\sgn}(\Phi_\lambda,\theta,\kappa,\chi,m)=\frac{\Omega(\Phi,\theta,\kappa)}
{\Omega^\sgn(f_\kappa,\chi)a_p(\theta,\kappa)^m}\Lambda(f_{\kappa,\lambda},\chi,m),$$
where $a_p(\theta,\kappa)^m:=\prod_{\frakp\mid
p}a_\frakp(\theta,\kappa)^{m_\frakp}$.\end{theorem}

\begin{proof} This follows by combining Proposition \ref{prop4.7}
with Proposition \ref{prop4.8}.\end{proof}

\begin{theorem}\label{th}
Notations and assumptions as in Proposition \ref{prop4.9}. Then
\begin{equation}L_p^\sgn(\Phi_\lambda,\theta,\kappa,\chi,m)=
e(p,\chi,m)
\frac{\Omega(\Phi,\theta,\kappa)}{\Omega^\sgn(f_{\kappa},\chi){a_p(\theta,\kappa)^m}}
\Lambda(f_{\kappa,\lambda},\chi,m).\end{equation}
\end{theorem}

\begin{proof} This follows from Proposition \ref{prop4.9} and Theorem
\ref{THM-1}.\end{proof}

Define $L_p^\sgn(\Phi,\theta,\kappa,\sigma)$ to be the $p$-adic
analytic function satisfying for any integer $m$ and any character
$\chi$ as above of sign $\sgn$:
\[L_p^\sgn(\Phi,\theta,\kappa,\chi\chi_\cyc^m)=\sum_{\lambda=1}^{h}
L_p^\sgn(\Phi_\lambda,\theta,\kappa,\chi\chi_\cyc^m)\norm{t_\lambda}^{-m+v+1}b_{v,\lambda}\]

\begin{theorem}\label{th-4.13}
Notations and assumptions as in Proposition \ref{prop4.9}. Then
\begin{equation}\label{int-th}L_p^\sgn(\Phi,\theta,\kappa,\chi\chi_\cyc^m)=
e(p,\chi,m)
\frac{\Omega(\Phi,\theta,\kappa)}{\Omega^\sgn(f_{\kappa},\chi){a_p(\theta,\kappa)^m}}
\Lambda(f_{\kappa},\chi,m).\end{equation}
\end{theorem}
\proof Follows from Theorem \ref{th} and Equation
\eqref{Lambda-1}.\endproof

\subsection{Relations with classical $p$-adic $L$-functions}

Recall Panchishkin's notion of ordinary form in \cite{Pa}. Let $g$
be a cusp form in $S_{k,w}(S(p^\alpha),\epsilon,\C)$ and write its
$\fr p$-Hecke polynomial for $\fr p\mid p$:
\[
1-C(\fr p,f)X+\epsilon(\fr p)\norm{\fr p}^{k_0-1}X^2=(1-\alpha(\fr
p)X)(1-\alpha^\prime(\fr p)X).
\] Order $\alpha(\fr p)$, $\alpha^\prime(\fr p)$ so that $\o_\fr p\alpha(\fr
p)\leq\o_\fr p\alpha^\prime(\fr p)$, where $\o_\fr p(p)$ is
normalized so that $\o_\fr p(p)=[F_\fr p:\Q_p]$. For any $\sigma\in
I$, denote by $\fr p(\sigma)$ the prime divisor of
$\iota_p\circ\mu(p)$.

\begin{defi}\label{pan-ord}
The modular form $g$ is said to be \emph{$p$-ordinary} if
\[
\o_\fr p\alpha(\fr p)=(k_0-k_\sigma)/2\] for all prime divisors $\fr
p$ of $p$, where $\sigma$ is chosen so that $\fr p=\fr p(\sigma)$ is
associated to $\sigma$.\end{defi}

It is clear that if $g$ is ordinary in the sense of Definition
\ref{pan-ord}, then it is nearly ordinary in the sense of Definition
\ref{def-ordinary}. For parallel weights, the two notions coincide.
For any ordinary eigenform $g$ (in the sense of Definition
\ref{pan-ord}) and any sign $\sgn$, denote by
$L_p^\sgn(g,\chi\chi_\cyc^m)$ the classical $p$-adic $L$-function
attached to $g$ constructed by Manin in \cite{Ma}. This function can
be characterized by requiring it to satisfy the following
interpolation property: for any integer $m$ in the critical strip of
$f$, \begin{equation}\label{int-formula} L_p^\sgn(f,m)=
\frac{e(p,\chi,m)\Lambda(f,\chi,m)}{\Omega^\sgn(f,\chi){a_p(\theta,\kappa)^m}}.\end{equation}
This formula can be found in a less explicit form in \cite[\S 5]{Ma}
or in \cite[Theorem 8.2]{Pa} in a form closest to this.

\begin{coro} Notations and assumptions as in Proposition \ref{prop4.9}. Suppose
further that  $f_\kappa$ is ordinary in the
sense of Definition \ref{pan-ord}. Then
\[L_p^\sgn(\Phi,\theta,\kappa,\sigma)=\Omega(\Phi,\theta,\kappa)L_p^\sgn(f_\kappa,\sigma).\]\end{coro}

\begin{proof}
This result follows by comparing the interpolation formulas
\eqref{int-th} and \eqref{int-formula}, after noticing that, by
\cite[Theorem 8.2 (iii)]{Pa}, $L_p^\sgn(f,\chi)$ is uniquely
determined by \eqref{int-formula}.\end{proof}

\end{document}